\documentclass[11pt]{article}

\usepackage{array}
\usepackage{amsfonts}
\usepackage{amscd}
\usepackage{amssymb}
\usepackage{amsthm}
\usepackage{amsmath}
\usepackage{stmaryrd}
\usepackage{graphicx}
\usepackage{color}
\usepackage{verbatim}
\usepackage{mathrsfs}
\usepackage{pstricks,tikz}
\usetikzlibrary{calc}
\usepackage{caption}
\usepackage{float}
\usepackage[neveradjust]{paralist}

 \theoremstyle{plain}
 \newtheorem{thm1}{Theorem}
 
  \newtheorem{prop1}[thm1]{Proposition}
\newtheorem{thm}{Theorem}[section]
\newtheorem{lemma}[thm]{Lemma}
\newtheorem{prop}[thm]{Proposition}
\newtheorem{cor}[thm]{Corollary}
\newtheorem{conj}[thm]{Conjecture}

\theoremstyle{definition}

\newtheorem{remark}[thm]{Remark}
\newtheorem{example}[thm]{Example}

\newtheorem*{question}{Question}
\numberwithin{equation}{section}

\def\fo{\mathfrak{o}}

\def\la{\lambda}
\def\vd{\vect{d}}

\def\E{X}%the arbitrary subset

\newcommand{\vect}[1]{\boldsymbol{#1}}

\def\fc{\mathfrak{c}}

\def\cA{\mathcal{A}}

\def\cF{\mathcal{F}}

\def\FF{\mathbb{F}}

\def\NN{\mathbb{N}}

\def\RR{\mathbb{R}}

\def\ZZ{\mathbb{Z}}

\DeclareMathOperator\wt{\mathrm{wt}}

\def\la{\lambda}

\def\<{\langle}
\def\>{\rangle}

\makeatletter
\renewcommand{\@makefnmark}{\mbox{\textsuperscript{}}}
\makeatother

\makeatletter
\def\keywords{\xdef\@thefnmark{}\@footnotetext}
\makeatother

\makeatletter
\def\MSC{\xdef\@thefnmark{}\@footnotetext}
\makeatother

\title{Patterns in sets of positive density in trees and affine buildings}
\author{M. Bj\"orklund, A. Fish, J. Parkinson}
\date{\today}

\begin{document}

\maketitle

\begin{abstract}
We prove an analogue for homogeneous trees and certain affine buildings of a result of Bourgain on pinned distances in sets of positive density in Euclidean spaces. Furthermore, we construct an example of a non-homogeneous tree with positive Hausdorff dimension, and a subset with positive density thereof, in which not all sufficiently large (even) distances are realised.
\end{abstract}
\MSC{2010 Mathematics Subject Classification: 05D10, 05C12, 05C42}

%\tableofcontents

\section*{Introduction}

 A celebrated result in geometric Ramsey theory due to Furstenberg, Katznelson and Weiss~\cite{FKW:89} states that if $\E$ is a Lebesgue measurable subset of $\RR^2$ with positive upper density then the set $D(\E)=\{\|x-y\|\mid x,y\in\E\}$ of distances between elements of $\E$ contains all sufficiently large real numbers. In \cite{Bou:86} Bourgain proved a remarkable generalisation of the Furstenberg-Katznelson-Weiss theorem, showing that if the measurable set $\E\subseteq \RR^n$ has positive upper density, and if~$\Sigma$ is the vertex set of an $(n-1)$-simplex in~$\RR^n$, then $\E$ contains an isometric copy of each sufficiently large dilation of~$\Sigma$. 

Discrete analogues of both the Furstenberg-Katznelson-Weiss and Bourgain theorems have recently been obtained by Magyar~\cite{Mag:08,Mag:09}. In \cite{Mag:08} it is shown that if $\E$ is a subset of positive upper density in $\ZZ^n$ with $n>4$, then the set $D^2(\E)=\{\|x-y\|^2\mid x,y\in\E\}$ contains all large multiples of a fixed integer $m$, and in \cite{Mag:09} an analogue of Bourgain's Theorem is proved. Further configurations in sets of positive density in $\ZZ^n$ have recently been studied in the papers~\cite{BF:16,BB:17}.

In this paper we investigate the extent to which certain configurations must necessarily be present in subsets of positive upper density in homogeneous trees and affine buildings. In particular we prove analogues of the Furstenberg-Katznelson-Weiss and Bourgain theorems in this ``$p$-adic'' context. 

Let us first discuss the case of trees. Let $T_{q}$ be the homogeneous tree of degree~$q+1\geq 3$ with vertex set $V$, and let $o\in V$ be a fixed choice of root. For $x\in V$ and $n\geq 0$ let $S_n(x)=\{y\in V\mid d(x,y)=n\}$, where $d(\cdot,\cdot)$ is the graph metric. The \textit{(upper) density} of a subset $\E\subseteq V$ is 
$$
d^*(\E)=\limsup_{n\to\infty}\frac{|\E\cap S_n(o)|}{|S_n(o)|}.
$$
Let $D(\E)=\{d(x,y)\mid x,y\in\E\}$ be the set of distances between elements of~$\E$. We prove the following analogue of the Furstenberg-Katznelson-Weiss theorem.

\begin{prop1}\label{prop:1}
Let $\E\subseteq V$ with $d^*(\E)>0$. There exits a constant $K=K(\E)>0$ such that for all $t\in\NN$ with $t\geq K$ there exist vertices $x,y\in \E$ with $d(x,y)=2t$ and $d(o,x)=d(o,y)$. %In particular, $\E\cap \E^{2t}\neq \emptyset$ for all $t>K$. 
\end{prop1}

Thus, in particular, if $d^*(\E)>0$ then $D(\E)$ contains all sufficiently large even integers. Note that since $T_q$ is bipartite it is obviously possible for $d^*(\E)>0$ and $D(\E)\subseteq 2\ZZ$, and so the condition of even distances in the proposition cannot be removed. Note also that we obtain the ``bonus'' that $x$ and $y$ are equidistant to $o$ in Proposition~\ref{prop:1}. This fact turns out to help with finding more elaborate configurations in sets of positive density in~$T_q$ (see Theorem~\ref{thm:main1} and Corollary~\ref{cor:elaborate}). 

The first main theorem of this paper is the following extension of Proposition~\ref{prop:1}, giving a kind of analogue of Bourgain's Theorem. 

\begin{thm1}\label{thm:main2}
Let $\E\subseteq V$ with $d^*(\E)>0$. For each $k>0$ there exists $K=K(\E,k)>0$ such that whenever $t_1,\ldots,t_k\in\NN$ with $t_k\geq \cdots\geq t_2\geq t_1\geq K$ there exists a subset $\{v_0,v_1,\ldots,v_k\}\subseteq \E$ with $d(v_0,v_j)=2t_j$ for all $1\leq j\leq k$ and $d(o,v_0)=d(o,v_1)=\cdots=d(o,v_k)$. 
\end{thm1}

Our second main theorem is an extension of Theorem~\ref{thm:main2} to sets of positive density in certain affine buildings (see Theorem~\ref{thm:main1buildings} for a precise statement). These combinatorial/geometric objects play a role for $p$-adic Lie groups analogous to the role that the symmetric space plays for a real Lie group. Homogeneous trees $T_q$ are the simplest types of affine buildings, being associated to the rank~$1$ group $SL_2(\mathbb{F})$ with $\mathbb{F}$ a local field with residue field $\FF_q$ (with $q$ a prime power).

Our proof techniques, for both the case of trees and buildings, are purely combinatorial, differing significantly from the ergodic theory techniques of \cite{FKW:89} and the harmonic analysis techniques of \cite{Bou:86,Mag:08,Mag:09}. This is not a matter of taste, but rather due to the very different types of symmetries of the patterns sought after in the two settings. More specifically, the type of patterns investigated in  \cite{Bou:86,FKW:89,Mag:08,Mag:09} are invariant under the action of the isometry group of the Euclidean space $\RR^n$, which is an \emph{amenable} group. The arguments in~\cite{FKW:89} make direct use of this fact, while the arguments~\cite{Bou:86,Mag:08,Mag:09} use it in an indirect way, to ultimately reduce the proofs to establishing vanishing of Bessel functions at infinity, which is classical.

On the other hand, the patterns investigated in this paper are invariant under the isometry group of $T_q$ (or a higher rank affine building), which is typically \emph{not} an amenable group. In particular, the approach of~\cite{FKW:89} is not applicable here, since it is not clear why the dynamical system (as described in~\cite{FKW:89}) attached to the set of positive density in the tree would admit a probability measure invariant under the action of the isometry group of $T_q$. If it did, then similar arguments to those in~\cite{FKW:89} could still be made, ultimately reducing the problem to establishing vanishing at infinity of the (positive definite) spherical functions of the tree, which is classical. However, this is a rather degenerate situation, and it is not hard to construct examples of sets with positive density in $T_q$ whose associated dynamical systems for the isometry group of $T_q$ do \emph{not} admit invariant probability measures. Fortunately, our combinatorial approach bypasses this problem entirely.

We conclude this introduction with an outline of the structure of the paper. In Section~\ref{sec:trees} we prove Proposition~\ref{prop:1} and Theorem~\ref{thm:main2}. In Section~\ref{sec:trees2} we address a related question asked to us by Itai Benjamini. In particular, we show that if $T\subseteq T_q$ is tree of positive Hausdorff dimension, and if $\E\subseteq T$ is a subset of positive lower density (hence also positive upper density) then the analogue of Proposition~\ref{prop:1} may fail. In Section~\ref{sec:buildings} we prove our extension of Theorem~\ref{thm:main2} for certain affine buildings (see Theorem~\ref{thm:main1buildings}), and we also translate our results to give a corollary on sets of positive density in $p$-adic Lie groups (see Corollary~\ref{cor:padic}). 

We note that Theorem~\ref{thm:main1buildings} (on affine buildings) covers Theorem~\ref{thm:main2} as the rank~$1$ case, and Theorem~\ref{thm:main2} in turn covers Proposition~\ref{prop:1} as a special case. Nonetheless we will provide complete proofs of both Proposition~\ref{prop:1} and Theorem~\ref{thm:main2} in this paper. We believe that this redundancy is well justified, as the tree case, and in particular Proposition~\ref{prop:1}, more clearly illustrates the key combinatorial ideas driving the proof of Theorem~\ref{thm:main1buildings}, yet avoids the technical complications encountered in the general case. Moreover, our decision to give a complete exposition of the tree case first makes our results more accessible to readers unacquainted with the theory of affine buildings.

%
%Finally, we note that our results can be regarded as a combinatorial version of mixing for homogeneous trees. \james{Michael and Sasha -- a paragraph and explanation about mixing here???} For $t\in\NN$ and subsets $\E\subseteq V$ let 
%$$
%\E^t=\{x\in \E\mid \text{there exists $y\in\E$ with $d(x,y)=t$}\}.
%$$
%We note the following immediate corollary to Theorem~\ref{thm:main2}. \james{this is meant to be related to mixing}
%
%
%\begin{cor1}
%Let $\E\subseteq V$ with $d^*(\E)>0$. For each $k>0$ there there exists $K>0$ such that for all distinct $t_1, t_2, \ldots, t_k>K$ the set $\E\cap \E^{2t_1}\cap\cdots\cap \E^{2t_k}$ is nonempty. 
%\end{cor1}

\section{Sets of positive density in homogeneous trees}\label{sec:trees}

Let $T_{q}$ be the homogeneous tree with vertex set $V$ and degree~$q+1\geq 3$, and let $o\in V$ be a fixed choice of root. For $x\in V$ and $n\geq 0$ let $S_n(x)=\{y\in V\mid d(x,y)=n\}$, where $d(\cdot,\cdot)$ is the graph metric. We write $S_n=S_n(o)$. %Note that $|S_n(x)|=(q+1)q^{n-1}$ for $n\geq 1$ is independent of $x\in V$.
The \textit{(upper) density} of a subset $\E\subseteq V$, with respect to $x\in V$, is 
$$
d^*(\E,x)=\limsup_{n\to\infty}\frac{|\E\cap S_n(x)|}{|S_n(x)|}.
$$
Let $d^*(\E)=d^*(\E,o)$. A subset $\E$ has \textit{positive density} if $d^*(\E)>0$. While the numerical value of density depends on the choice of root, the property of positive density is independent of this choice, as shown by the following lemma.

\begin{lemma}
If $d^*(\E,x)> 0$ for some $x\in V$ then $d^*(\E,x)> 0$ for all $x\in V$. 
\end{lemma}

\begin{proof}
It suffices to show that $d^*(\E,y)>0$ whenever $d(x,y)=1$. For $n\geq 1$ we have $S_n(x)\subseteq S_{n-1}(y)\cup S_{n+1}(y)$ and hence 
$$
\frac{|\E\cap S_n(x)|}{|S_n(x)|}\leq \frac{|\E\cap S_{n-1}(y)|}{|S_n(y)|}+\frac{|\E\cap S_{n+1}(y)|}{|S_n(y)|}\leq (q+q^{-1})\left(\frac{|\E\cap S_{n-1}(y)|}{|S_{n-1}(y)|}+\frac{|\E\cap S_{n+1}(y)|}{|S_{n+1}(y)|}\right)
$$
Thus
$
d^*(\E,x)\leq 2(q+q^{-1})d^*(\E,y).
$ Thus for arbitrary $y\in V$ we have $d^*(\E,y)\geq C^{-d(x,y)}d^*(\E,x)$, where $C=2(q+q^{-1})$, and hence the result.
\end{proof}

\begin{remark} There exist subsets $\E\subseteq V$ of positive density with $\inf\{d^*(\E,x)\mid x\in V\}=0$ and $\sup\{d^*(\E,x)\mid x\in V\}=1$. For example, if $\E$ consists of one entire ``branch'' based at $o$ then $d^*(\E)=(q+1)^{-1}>0$, and choosing sequences $(x_n)_{\geq 0}$ and $(y_n)_{n\geq 0}$ of vertices with $x_n\to\infty$, $y_n\to\infty$, $x_n\in \E$, and $y_n\notin \E$, we have $d^*(\E,x_n)\to 1$ and $d^*(\E,y_n)\to 0$. This example illustrates that the constant $K$ appearing in Theorem~\ref{thm:main2} must depend on the set $\E$, rather than depending only on $d^*(\E)$.
\end{remark}

\subsection{Proof of Proposition~\ref{prop:1}}

In this section we give a proof of Proposition~\ref{prop:1}, illustrating the proof techniques required for Theorem~\ref{thm:main2} in a simplified setting. If $v\in V$ and $k\in\mathbb{N}$ we write $C(v,k)$ for the ``$k$-children'' (or $k$-descendants) of $v$. That is,
$$
C(v,k)=\{x\in V\mid d(v,x)=k\text{ and }d(o,x)=d(o,v)+k\}.
$$
Let $C(v)$ denote the set of all decendants of $v$. That is,
$$
C(v)=\bigcup_{k\geq 0}C(v,k).
$$
For each $n\geq 0$ and $0\leq t\leq n$ the members of $\cA_{n,t}=\{C(v,t)\mid v\in S_{n-t}\}$ partition $S_n$. Let $\cF_{n,t}$ denote the $\sigma$-algebra generated $\cA_{n,t}$. We call the members of $\cA_{n,t}$ the \textit{atoms} of $\cF_{n,t}$.

\begin{proof}[Proof of Proposition~\ref{prop:1}]
The argument proceeds as follows. 
\smallskip

\noindent\textit{Claim 1:} Suppose that $t_1\in\NN$ is such that there exist no vertices $x,y\in\E$ with $d(o,x)=d(o,y)$ and $d(x,y)=2t_1$. Then for each $v\in S_{n-t}$ (with $n> t\geq t_1$) the proportion of $\cF_{n,t_1-1}$ atoms contained in $C(v,t)$ with the property that they intersect nontrivially with $\E$ is at most~$q^{-1}$.
\smallskip

\noindent\textit{Proof of Claim 1:} Let $v\in S_{n-t}$. We decompose $C(v,t)$ as 
$
C(v,t)=\bigsqcup_{w\in C(v)\cap S_{n-t_1}}C(w,t_1).
$
Let $w\in C(v)\cap S_{n-t_1}$. If there exist distinct $u_1,u_2\in C(w,1)$ such that $C(u_j,t_1-1)\cap \E\neq \emptyset$ for $j=1,2$, then choosing $z_j\in C(u_j,t_1-1)\cap \E$ (for $j=1,2$) we have $d(o,z_1)=n=d(o,z_2)$, and $d(z_1,z_2)=2t_1$, a contradiction (see Figure~\ref{fig:claim1}, where the $\cF_{n,t_1-1}$ atoms contained in $C(v,n)$ are drawn as ellipses, and nontrivial intersection of an atom with $\E$ is denoted by shading). Thus for each $w\in C(v)\cap S_{n-t_1}$ there exists at most $1$ child $u\in C(w,1)$ with the property that $C(u,t_1-1)$ intersects nontrivially with $\E$ (see the elements $w'$ and $u'$ in Figure~\ref{fig:claim1}). Hence the claim.
\begin{figure}[H]
\begin{center}
\begin{tikzpicture}[xscale=0.5, yscale=0.25]
\draw [fill=gray!70!] (-7,4) ellipse (1.5cm and 0.5cm);
\draw (3,4) ellipse (1.5cm and 0.5cm);
\draw [fill=gray!70!] (7,4) ellipse (1.5cm and 0.5cm);
\draw [fill=gray!70!] (-3,4) ellipse (1.5cm and 0.5cm);
%\node at (-4,4) {$\bullet$};
\node at (-9.5,10) [left] {$n-t_1+1$};
\node at (-9.5,4) [left] {$n$};
\node at (-9.5,16) [left] {$n-t$};
\node at (0,16) [above] {$v$};
\node at (-9.5,12) [left] {$n-t_1$};
\node at (-5,12) [above] {$w$};
%\node at (6.5,10) [below right] {$y$};
\node at (5,12) [above] {$w'$};
\node at (-3.5,10) [below left] {$u_2$};
\node at (-7.5,10) [below left] {$u_1$};
\node at (7.5,10) [right] {$u'$};
\draw (-5,12)--(0,16)--(5,12);
%\node at (6.5,0) [right, above] {$x$};
\draw (-5,12)--(-7,10);
\draw (-5,12)--(-3,10);
\draw (-7,10)--(-8.5,4);
\draw (-7,10)--(-5.5,4);
\draw (-3,10)--(-4.5,4);
\draw (-3,10)--(-1.5,4);
%\draw [style=dotted] (-6,10) to [bend right=30] (-4,10);
%\draw [style=dotted] (-7.5,7) to [bend right=30] (-6.5,7);
%\draw [style=dotted] (-3.5,7) to [bend right=30] (-2.5,7);
%
\draw (5,12)--(7,10);
\draw (5,12)--(3,10);
\draw (7,10)--(8.5,4);
\draw (7,10)--(5.5,4);
\draw (3,10)--(4.5,4);
\draw (3,10)--(1.5,4);
%\draw [style=dotted] (6,10) to [bend left=30] (4,10);
%\draw [style=dotted] (7.5,7) to [bend left=30] (6.5,7);
%\draw [style=dotted] (3.5,7) to [bend left=30] (2.5,7);
%\draw [style=dotted] (-9,-3)--(9,-3);
\draw [style=dotted] (-9,4)--(-8.5,4);
\draw [style=dotted] (-9,12)--(-5,12);
\draw [style=dotted] (-9,10)--(-7,10);
\draw [style=dotted] (-9,16)--(0,16);
\node at (-2.5,4) {$\bullet$};
\node at (-7.5,4) {$\bullet$};
\node at (7.5,4) {$\bullet$};
\node at (7,4) {$\bullet$};
\node at (6.5,4) {$\bullet$};
\node at (8,4) {$\bullet$};
\node at (6,4) {$\bullet$};
\node at (-2.5,2.5) {$z_2\in \E$};
\node at (-7.5,2.5) {$z_1\in \E$};
\phantom{\node at (12,5) {$q$};}
%\node at (7,2.5) {$\E\cap C(x,s)\subseteq C(y,s-1)$};
\end{tikzpicture}
\end{center}
\caption{Illustration for Claim 1}
\label{fig:claim1}
\end{figure}
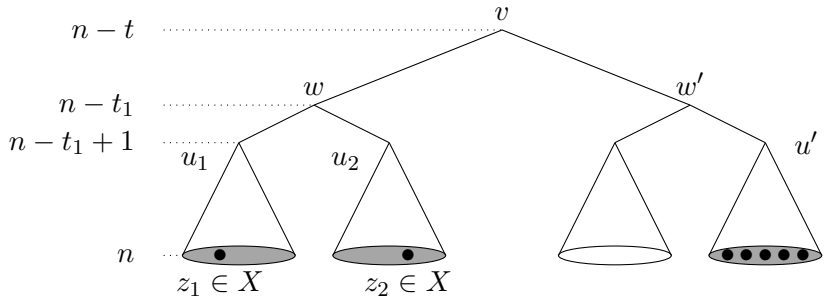

\noindent\textit{Claim 2:} Suppose there are integers $0<t_1<t_2<\cdots<t_k$ such that for each $1\leq j\leq k$ there exist no vertices $x,y\in\E$ with $d(o,x)=d(o,y)$ and $d(x,y)=2t_j$. Then for all $n>t_k$ we have 
$$
\frac{|\E\cap S_n|}{|S_n|}<q^{-k}.
$$
\smallskip

\noindent\textit{Proof of Claim 2:} Let $v_k\in S_{n-t_k}$. By Claim 1 (with $v=v_2$), the proportion of atoms of $\cF_{n,t_k-1}$ contained in $C(v_k,t_k)$ with the property that they intersect nontrivially with $\E$ is at most~$q^{-1}$. However $C(v_k,t_k)$ contains precisely $q$ atoms of $\cF_{n,t_k-1}$, and hence there is at most one $\cF_{n,t_k-1}$ atom contained in $C(v_k,t_k)$ with the property that it intersects nontrivially with $\E$. Suppose that~$A_k$ is such an atom, illustrated as a dashed ellipse in Figure~\ref{fig:claim2} (for the case $k=2$). 
\begin{figure}[H]
\begin{center}
\begin{tikzpicture}[xscale=0.5, yscale=0.25]
\draw [style=dashed] (0,4) ellipse (7.5cm and 2cm);
%\draw [fill=white] (-6,4) ellipse (1.5cm and 0.5cm);
%\draw [fill=white] (2,4) ellipse (1.5cm and 0.5cm);
%\draw [fill=white] (6,4) ellipse (1.5cm and 0.5cm);
%\draw [fill=white] (-2,4) ellipse (1.5cm and 0.5cm);
%
\draw [fill=gray!70!] (-6,4) ellipse (1.5cm and 0.5cm);
\draw (2,4) ellipse (1.5cm and 0.5cm);
\draw [fill=gray!70!] (6,4) ellipse (1.5cm and 0.5cm);
\draw (-2,4) ellipse (1.5cm and 0.5cm);
%\node at (-4,4) {$\bullet$};
\node at (-9.5,10) [left] {$n-t_1+1$};
\node at (-9.5,4) [left] {$n$};
\node at (-9.5,18) [left] {$n-t_2$};
\node at (-9.5,16) [left] {$n-t_2+1$};
\node at (0,16) [above] {$v_2'$};
\node at (4,18) [above] {$v_2$};
\node at (-9.5,12) [left] {$n-t_1$};
%\node at (-5,12) [above] {$w$};
%\node at (6.5,10) [below right] {$y$};
%\node at (5,12) [above] {$w'$};
%\node at (-3.5,10) [below left] {$u_2$};
%\node at (-7.5,10) [below left] {$u_1$};
%\node at (7.5,10) [right] {$u'$};
\draw (-4,12)--(0,16)--(4,12);
%\node at (6.5,0) [right, above] {$x$};
\draw (-4,12)--(-6,10);
\draw (-4,12)--(-2,10);
\draw (-6,10)--(-7.5,4);
\draw (-6,10)--(-4.5,4);
\draw (-2,10)--(-3.5,4);
\draw (-2,10)--(-0.5,4);
%\draw [style=dotted] (-6,10) to [bend right=30] (-4,10);
%\draw [style=dotted] (-7.5,7) to [bend right=30] (-6.5,7);
%\draw [style=dotted] (-3.5,7) to [bend right=30] (-2.5,7);
%
\draw (4,12)--(6,10);
\draw (4,12)--(2,10);
\draw (6,10)--(7.5,4);
\draw (6,10)--(4.5,4);
\draw (2,10)--(3.5,4);
\draw (2,10)--(0.5,4);
%\draw [style=dotted] (6,10) to [bend left=30] (4,10);
%\draw [style=dotted] (7.5,7) to [bend left=30] (6.5,7);
%\draw [style=dotted] (3.5,7) to [bend left=30] (2.5,7);
%\draw [style=dotted] (-9,-3)--(9,-3);
\draw [style=dotted] (-9,4)--(-7.5,4);
\draw [style=dotted] (-9,12)--(-4,12);
\draw [style=dotted] (-9,10)--(-6,10);
\draw [style=dotted] (-9,16)--(0,16);
\draw [style=dotted] (-9,18)--(4,18);
%\node at (-2.5,4) {$\bullet$};
%\node at (-7.5,4) {$\bullet$};
%\node at (7.5,4) {$\bullet$};
%\node at (7,4) {$\bullet$};
%\node at (6.5,4) {$\bullet$};
%\node at (8,4) {$\bullet$};
%\node at (6,4) {$\bullet$};
%\node at (-2.5,2.5) {$z_2\in \E$};
\node at (0,3) {$A_2$};
\draw (0,16)--(4,18);
\draw (4,18)--(7,16);
\draw [style=dashed] (7,16)--(5,14);
\draw [style=dashed] (7,16)--(9,14);
%\node at (7,2.5) {$\E\cap C(x,s)\subseteq C(y,s-1)$};
\phantom{\node at (13,5) {$q$};}
\end{tikzpicture}
\end{center}
\caption{Illustration for Claim~$2$ (with $k=2$)}
\label{fig:claim2}
\end{figure}
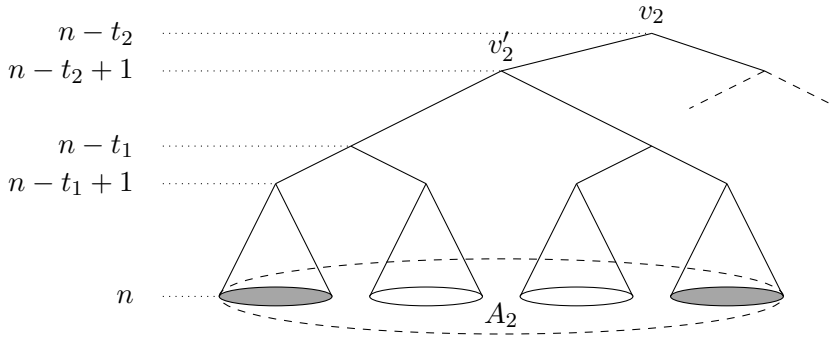
\noindent Let $v_k'$ be such that $A_k=C(v')\cap S_n$, as illustrated. Then, again by Claim 1 (this time with $v=v_2'$), the proportion of the $\cF_{n,t_{k-1}-1}$ atoms contained in $A_k$ with the property that they intersect nontrivially with $\E$ is at most $q^{-1}$ (these atoms are displayed as shaded ellipses in Figure~\ref{fig:claim2}). Thus, overall, the proportion of the $\cF_{n,t_{k-1}-1}$ atoms contained in $C(v_k,t_k)$ with the property that they intersect nontrivially with $\E$ is at most~$q^{-2}$. Iterating this process proves the claim.

\smallskip

The proposition now follows. For if there is an unbounded sequence $t_1<t_2<t_3<\cdots$ such that for each $j\geq 1$ there are no vertices $x,y\in \E$ with $d(o,x)=d(o,y)$ such that $d(x,y)=2t_j$ then for each $k>0$ and each $n>t_k$ we have $|\E\cap S_n|/|S_n|<q^{-k}$. It follows that $d^*(\E)=0$, contradicting positive density. 
\end{proof}

\subsection{Proof of Theorem~\ref{thm:main1}}

The arguments in the previous section are the core to the proof of Theorem~\ref{thm:main2}, however the details become more technical in the general case. We introduce the following notion for the proof. For each $k>0$ let $\mathbb{M}^k$ denote the set of strictly monotone increasing sequences of $k$ positive integers. Let $k>0$. Let $\vect{t}=(t_i)_{i=1}^k\in\mathbb{M}^k$ and let $\vect{r}=(r_i)_{i=1}^k\in\mathbb{N}^k$. Then by a $(k,\vect{t},\vect{r})$-star we mean a set $Y$ of vertices of $T_q$ with a distinguished vertex $v_0\in Y$ (called the \textit{centre} of $Y$) such that
\begin{compactenum}[$(1)$]
\item if $v\in Y$ then $d(v_0,v)\in\{2t_1,\ldots,2t_k\}$,
\item for each $1\leq i\leq k$ we have $|\{v\in Y\mid d(v_0,v)=2t_i\}|=r_i$.
\end{compactenum}
In particular, note that a $(k,\vect{t},\vect{r})$-star has precisely $1+r_1+\cdots+r_k$ vertices. 

We call a $(k,\vect{t},\vect{r})$-star $Y$ \textit{balanced} if $d(o,x)$ is constant for all $x\in Y$ (that is, $Y\subseteq S_n$ for some $n>0$). If $Y$ is balanced, then writing $Y_i=\{x\in Y\mid d(v_0,x)=2t_i\}$ for $i=0,1,\ldots,k$ we have $Y=Y_0\cup Y_1\cup \cdots\cup Y_k$ and $d(x,y)=2t_j$ for all $x\in Y_i$ and $y\in Y_j$ whenever $0\leq i<j\leq k$. This configuration is illustrated below.
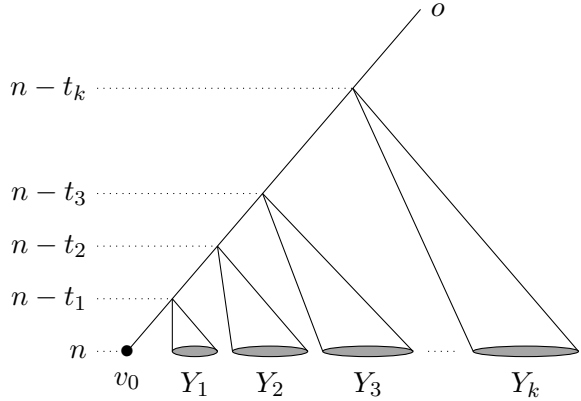
\begin{figure}[H]
\begin{center}
\begin{tikzpicture}[xscale=0.8, yscale=0.7]
\draw [fill=gray!70!] ({-2.25+(-1.5+2.25)/2},-4) ellipse (0.375cm and 0.1cm);
\draw [fill=gray!70!] ({-1.25+(0+1.25)/2},-4) ellipse (0.625cm and 0.1cm);
\draw [fill=gray!70!] ({0.25+(1.75-0.25)/2},-4) ellipse (0.75cm and 0.1cm);
\draw [fill=gray!70!] ({0.25+2.5+(4.25-2.5)/2},-4) ellipse (0.875cm and 0.1cm);
\draw (-3,-4)--(1.875,2.5);
\draw (-2.25,-3)--(-2.25,-4);%X1
\draw (-2.25,-3)--(-1.5,-4);
\draw (-1.5,-2)--(-1.25,-4);%X2
\draw (-1.5,-2)--(0,-4);
\draw (-0.75,-1)--(0.25,-4);%X3
\draw (-0.75,-1)--(1.75,-4);
\draw ({0.75},1)--({2.5+0.25},-4);%Xk
\draw ({0.75},1)--({4.25+0.25},-4);
\draw [style=dotted] (-3.5,-4)--(-3,-4);
\draw [style=dotted] (-3.5,-3)--(-2.25,-3);
\draw [style=dotted] (-3.5,-2)--(-1.5,-2);
\draw [style=dotted] (-3.5,-1)--(-0.75,-1);
\draw [style=dotted] (-3.5,1)--(0.75,1);
\node at (-3,-4) {$\bullet$};
\node at (-3,-4.2) [below] {$v_0$};
\node at ({-2.25+(-1.5+2.25)/2},-4.2) [below] {$Y_1$};
\node at ({-1.25+(0+1.25)/2},-4.2) [below] {$Y_2$};
\node at ({0.25+(1.75-0.25)/2},-4.2) [below] {$Y_3$};
\node at ({0.25+2.5+(4.25-2.5)/2},-4.2) [below] {$Y_k$};
\node at (1.875,2.5) [right] {$o$};
\node at (-3.5,-4) [left] {$n$};
\node at (-3.5,-3) [left] {$n-t_1$};
\node at (-3.5,-2) [left] {$n-t_2$};
\node at (-3.5,-1) [left] {$n-t_3$};
\node at (-3.5,1) [left] {$n-t_k$};
\draw [style=dotted] (2,-4)--(2.5,-4);
\end{tikzpicture}
\end{center}
\caption{A balanced $(k,\vect{t},\vect{r})$-star.}
\label{fig:configuration}
\end{figure}

Theorem~\ref{thm:main2} follows immediately from the following theorem.

\begin{thm}\label{thm:main1}
Let $\E\subseteq V$ with $d^*(\E)>0$. For each $k>0$ and each $\vect{r}\in \mathbb{N}^k$ there exists a constant $K=K(\E,k,\vect{r})>0$ such that $\E$ contains a balanced $(k,\vect{t},\vect{r})$-star for all sequences $\vect{t}=(t_i)_{i=1}^k\in\mathbb{M}^k$ with $t_1>K$. 
\end{thm}

%Before giving the proof of the theorem we note the following immediate corollary. 
%
%If $\E\subseteq V$, for each $t\in\mathbb{N}$ write
%$$
%\E^t=\{x\in \E\mid \text{there exists $y\in \E$ with $d(x,y)=t$}\}.
%$$

%
%\begin{cor}
%Let $\E\subseteq V$ with $d^*(\E)>0$. For each $k>0$ there there exists $K>0$ such that for all distinct $t_1, t_2, \ldots, t_k>K$ the set $\E\cap \E^{2t_1}\cap\cdots\cap \E^{2t_k}$ is nonempty. 
%\end{cor}
%
%\begin{proof}
%This follows from Theorem~\ref{thm:main1}, taking $\vect{r}=(1,1,\ldots,1)$.
%\end{proof}
%

\goodbreak

To prove Theorem~\ref{thm:main1} we argue as in Proposition~\ref{prop:1}, using the following lemmas. 

\begin{lemma}\label{lem:1}
Let $\E\subseteq V$, $k>0$, $\vect{r}\in\mathbb{N}^k$, and $\vect{t}\in\mathbb{M}^k$. Let $r=\max\{r_i\mid 1\leq i\leq k\}$. If $\E$ contains no balanced $(k,\vect{t},\vect{r})$-star then for each $v\in S_{n-t}$ with $n\geq t\geq t_k$ the proportion of the atoms of $\cF_{n,t_1-1}$ contained in $C(v,t)$ with the property that they intersect $\E$ in at least $r$ vertices is at most~$q^{-1}$.
\end{lemma}

\begin{proof}
We introduce the following terminology for the proof. A vertex $x\in S_{n-s}$ (with $0\leq s\leq n$) is said to have ``type $A$'' if there are at least two children  $y_1,y_2\in C(x,1)$ with the property that $\E\cap C(y_j,s-1)$ contains at least $r$ elements for $j=1,2$, and is said to have ``type $B$'' otherwise. Figure~\ref{fig:types} illustrates a type $A$ vertex $x$ (with $r\leq 4$, and where elements of $\E$ are denoted by $\bullet$). 
\begin{figure}[H]
\begin{center}
\begin{tikzpicture}[xscale=0.5, yscale=0.3]
%\draw [fill=gray!70!] (-7,4) ellipse (1.5cm and 0.5cm);
%\draw (-3,4) ellipse (1.5cm and 0.5cm);
\draw [fill=gray!70!] (7,4) ellipse (1.5cm and 0.5cm);
\draw [fill=gray!70!] (3,4) ellipse (1.5cm and 0.5cm);
%\node at (-4,4) {$\bullet$};
\node at (0,9) [left] {$n-s+1$};
\node at (0,4) [left] {$n$};
\node at (0,12) [left] {$n-s$};
%\node at (-5,12) [above] {$x$};
%\node at (-6.5,9) [below right] {$y$};
\node at (5,12) [above] {$x$};
\node at (3.5,9) [below right] {$y_1$};
\node at (7.5,9) [below right] {$y_2$};
%\node at (6.5,0) [right, above] {$x$};
%\draw (-5,12)--(-7,9);
%\draw (-5,12)--(-3,9);
%\draw (-7,9)--(-8.5,4);
%\draw (-7,9)--(-5.5,4);
%\draw (-3,9)--(-4.5,4);
%\draw (-3,9)--(-1.5,4);
%\draw [style=dotted] (-6,10) to [bend right=30] (-4,10);
%\draw [style=dotted] (-7.5,7) to [bend right=30] (-6.5,7);
%\draw [style=dotted] (-3.5,7) to [bend right=30] (-2.5,7);
%
\draw (5,12)--(7,9);
\draw (5,12)--(3,9);
\draw (7,9)--(8.5,4);
\draw (7,9)--(5.5,4);
\draw (3,9)--(4.5,4);
\draw (3,9)--(1.5,4);
%\draw [style=dotted] (6,10) to [bend left=30] (4,10);
%\draw [style=dotted] (7.5,7) to [bend left=30] (6.5,7);
%\draw [style=dotted] (3.5,7) to [bend left=30] (2.5,7);
%\draw [style=dotted] (-9,-3)--(9,-3);
\draw [style=dotted] (0,4)--(1.5,4);
\draw [style=dotted] (0,12)--(5,12);
\draw [style=dotted] (0,9)--(3,9);
\node at (2.25,4) {$\bullet$};
\node at (2.75,4) {$\bullet$};
\node at (3.25,4) {$\bullet$};
\node at (3.75,4) {$\bullet$};
\node at (6.25,4) {$\bullet$};
\node at (6.75,4) {$\bullet$};
\node at (7.25,4) {$\bullet$};
\node at (7.75,4) {$\bullet$};
%\node at (-7.5,4) {$\bullet$};
%\node at (-7,4) {$\bullet$};
%\node at (-6.5,4) {$\bullet$};
%\node at (-8,4) {$\bullet$};
%\node at (-6,4) {$\bullet$};
%\node at (5,2.5) {$\bullet\in \E$};
%\node at (10,4) {$\bullet\in \E$};
%\node at (8,2.5) {$\bullet\in \E$};
%\node at (-7,2.5) {$\E\cap C(x,s)\subseteq C(y,s-1)$};
\end{tikzpicture}
\end{center}
\caption{The vertex $x$ has type $A$.}
\label{fig:types}
\end{figure}
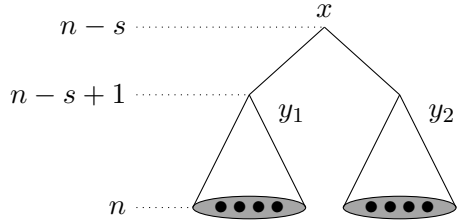

We make the following observation. Let $0<s' < s<n$. If $x\in S_{n-s}$ then 
\begin{align}\label{eq:treedecomposition}
C(x,s)=\bigsqcup_{y\in C(x,1)}C(y,s-1),
\end{align}
and for each $y\in C(x,1)$ the set $C(y,s-1)$ is a union of $q^{s-s'-1}$ atoms of $\cF_{n,s'}$. In particular, each set $C(y,s-1)$, $y\in C(x,1)$, contains the same number of $\cF_{n,s'}$ atoms, and so if $x\in S_{n-s}$ has type $B$ then the proportion of the $\cF_{n,s'}$ atoms of $C(x,s)$ containing at least $r$ elements of $\E$ is at most~$q^{-1}$.

Returning to the argument, let $v\in S_{n-t}$. Let $\mathcal{X}$ be the set of all sequences $(x_k,x_{k-1},\ldots,x_1)$ such that $x_k\in S_{n-t_k}\cap C(v)$ and $x_{k-j}\in C(x_{k-j+1})\cap S_{n-t_{k-j}}$ for $j=1,2,\ldots,k-1$. Suppose that there exists a sequence $(x_k,x_{k-1},\ldots,x_1)\in\mathcal{X}$ such that each $x_j$ has type~$A$. So there exist vertices $y_1,y_1'\in C(x_1,1)$ such that $|\E\cap C(y_1,t_1-1)|\geq r$ and $|\E\cap C(y_1',t_1-1)|\geq r$. Let $v_0\in C(y_1,t_1-1)\cap \E$ and $v_{1}^1,\ldots,v_{1}^r$ be distinct elements of $C(y_1',t_1-1)\cap \E$. For each $j=2,\ldots,k$ let $y_j\in C(x_j,1)$ denote the unique child of $x_j$ on the geodesic segment joining $x_k$ to $x_1$. Since $x_j$ has type~$A$ there is a second vertex $y_j'\neq y_j$ in $C(x_j,1)$ such that $|\E\cap C(y_j',t_j-1)|\geq r$. Let $v_j^1,\ldots,v_j^r$ be distinct elements of $\E\cap C(y_j',t_j-1)$. Then $d(v_0,v_i^j)=2t_i$ for each $1\leq i\leq k$ and $1\leq j\leq r$, and so $\{v_0\}\cup\{v_i^j\mid 1\leq i\leq k,\,1\leq j\leq r\}$ forms a balanced $(k,\vect{t},r\vect{1})$-star, where $\vect{1}\in\mathbb{N}^k$ is the vector with every entry equal to~$1$. In particular $\E$ contains a balanced $(k,\vect{t},\vect{r})$-star, a contradiction.

Thus every $(x_k,x_{k-1},\ldots,x_1)\in\mathcal{X}$ contains at least one vertex of type~$B$. Note that if $x_j$ has type~$B$ then the proportion of $\cF_{n,t_1-1}$ atoms of $C(x_j,t_j)$ intersecting $\E$ in at least $r$ vertices is at most~$q^{-1}$. Thus by a depth-first scan through the natural forest structure on $\mathcal{X}$ (with root nodes $x_k\in S_{n-t_k}\cap C(v)$) we can partition the set of $\cF_{n,t_1-1}$ atoms in $C(v,t)$ in such a way that in each part of the partition the proportion of atoms with the property that they intersect $\E$ in at least $r$ vertices is at most~$q^{-1}$. Thus the proportion of all $\cF_{n,t_1-1}$ atoms of $C(v,t)$ with this property is at most~$q^{-1}$, and hence the result.
\end{proof}

\begin{lemma}\label{lem:2}
Let $\E\subseteq V$, $k>0$, and $\vect{r}\in\mathbb{N}^k$. For $1\leq j\leq \ell$ let $\vect{t}_j=(t_{i,j})_{i=1}^k\in \mathbb{M}^k$, and suppose that $t_{k,j}<t_{1,j+1}$ for each $j=1,\ldots,\ell-1$. If $\E$ contains no balanced $(k,\vect{t}_j,\vect{r})$-stars for each  $1\leq j\leq \ell$ then for all $n>t_{k,\ell}$ we have
$$
\frac{|\E\cap S_n|}{|S_n|}<(q^{-\ell}+rq^{1-t_{1,1}}),
$$
where $r=\max\{r_i\mid 1\leq i\leq k\}$. 
\end{lemma}

\begin{proof}
Lemma~\ref{lem:1} (applied to the case $v=o$) implies that the proportion of $\cF_{n,t_{1,\ell}-1}$ atoms intersecting $\E$ in at least $r$ vertices is at most $q^{-1}$. Let $X_{\ell}$ be such an atom, and let $x_{\ell}$ be the projection of this atom onto $S_{n-t_{1,\ell}+1}$ (that is, $X_{\ell}=S_n\cap C(x_{\ell})$). Lemma~\ref{lem:1}, this time applied to the case $v=x_{\ell}$, implies that the proportion of the $\cF_{n,t_{1,\ell-1}-1}$ atoms contained in $X_{\ell}$ with the property that they intersect $\E$ in at least $r$ vertices is at most~$q^{-1}$. Hence the proportion of all $\cF_{n,t_{1,\ell-1}-1}$ atoms with the property that they intersect $\E$ in at least $r$ vertices is at most~$q^{-2}$. Iterating this process shows that the proportion of all $\cF_{n,t_{1,1}-1}$ atoms with the property that they intersect $\E$ in at least $r$ vertices is at most~$q^{-\ell}$. Each atom in the remaining $(1-q^{-\ell})$ proportion of atoms contains at most $r-1$ elements of $\E$. Since the total number of $\cF_{n,t_{1,1}-1}$ atoms is $(q+1)q^{n-t_{1,1}}=q^{1-t_{1,1}}|S_n|$ we have
\begin{align*}
|\E\cap S_n|&<q^{-\ell}|S_n|+(1-q^{-\ell})q^{1-t_{1,1}}r|S_n|,
\end{align*}
hence the result.
\end{proof}

%
%For $\E\subseteq V$ we define the density by
%\begin{align*}
%d^*(\E)=\limsup_{n\to\infty}\frac{1}{n}\sum_{k=0}^n\frac{|S_n\cap \E|}{|S_n|}.
%\end{align*}
%(we might use $\liminf$ instead I guess).
%

\begin{proof}[Proof of Theorem~\ref{thm:main1}]
Suppose not. Then there exists an integer $k>0$ and a vector $\vect{r}\in\mathbb{N}^k$ such that for all integers $K>0$ there is $\vect{s}=(s_i(K))_{i=1}^k\in\mathbb{M}^k$ with $s_1(K)>K$ such that $\E$ contains no balanced $(k,\vect{s},\vect{r})$-star. Let $\ell>0$ be any integer. Recursively define integers $t_{i,j}$, for $1\leq i\leq k$ and $1\leq j\leq \ell$, by setting $t_{i,1}=s_i(\ell)$ for all $1\leq i\leq k$ and
$$
t_{i,j}=s_i(t_{i,j-1})\quad\text{for $1\leq i\leq k$ and $2\leq j\leq \ell$}.
$$
Let $\vect{t}_j=(t_{i,j})_{i=1}^k$. These sequences satisfy the hypothesis of Lemma~\ref{lem:2}, and since $t_{1,1}=s_1(\ell)>\ell$ we have 
\begin{align*}
\frac{|\E\cap S_n|}{|S_n|}< q^{-\ell}(1+rq)
\end{align*}
for all sufficiently large~$n$, where $r=\max\{r_i\mid 1\leq i\leq k\}$. Thus $d^*(\E)\leq q^{-\ell}(1+rq)$ for each $\ell>0$, contradicting positive density. 
\end{proof}

\subsection{More elaborate configurations}

Theorem~\ref{thm:main1} implies that we can find any star configuration in a dense subset of a tree, provided that all of the distances to the centre of the star are large enough. It is natural to try to go further and add edges to the star configuration (hence making a tree~$T$), and asking if we can again find such configurations in a dense subset of~$T_q$. Indeed we can, provided we are just interested in the distances between adjacent vertices of this tree~$T$. For a precise statement, we need the following definitions.

Let $T$ be a finite rooted tree with vertex set $V(T)=\{0,1,\ldots,N\}$, with $0$ being the root. Write $E(T)$ for the set of undirected edges of $T$. Let $\wt:E(T)\to \ZZ_{>0}$ be a weight function. We consider the weight function as a function on $V(T)\backslash\{0\}$ by setting $\wt(j)=\wt(\{i,j\})$ where $i\in V(T)$ is the unique vertex of $T$ with $d_T(i,j)=1$ and $d_T(0,i)=d_T(0,j)-1$ (that is, $i$ is the penultimate vertex on the geodesic from $0$ to $j$). We say that the weight function $\wt:E(T)\to \ZZ_{>0}$ is \textit{well ordered} if $\wt(j)<\wt(k)$ whenever $d_T(0,j)<d_T(0,k)$, and we say that the weight function is \textit{bounded below} by $K$ if $\min\{\wt(j)\mid j\in\{1,2,\ldots,N\}\}\geq K$

\begin{cor}\label{cor:elaborate}
Let $\E$ be a set of positive density in $T_q$, and let $T$ be a finite rooted tree with vertex set $V(T)=\{0,1,\ldots,N\}$ and root~$0$. There exists $K=K(T,\E)$ such that for each choice of well ordered weight function $\wt:E(T)\to\ZZ_{>0}$ on $T$ bounded below by $K$ there exists a subset $\{v_0,\ldots,v_N\}\subseteq \E$ such that $d(v_i,v_j)=2\wt(j)$ whenever $d_T(0,i)<d_T(0,j)$, and moreover $d(o,v_0)=d(o,v_1)=\cdots=d(o,v_N)$. 
\end{cor}

\begin{proof}
Let $\vect{t}$ be the vector of all distinct edge weights of $T$, arranged in increasing order. Let $\vect{r}$ denote the vector of multiplicities of the edge weights (that is, if the edge weight $t$ appears $r$ times in $T$ then the entry of $\vect{r}$ corresponding to the entry $t$ of $\vect{t}$ is~$r$). It follows from Theorem~\ref{thm:main1} there exists $K=K(T,\E)$ such that whenever $\wt:E(T)\to\ZZ_{>0}$ is bounded below by $K$ there exists a balanced $(k,\vect{t},\vect{r})$-star.  The result follows (see Figure~\ref{fig:configuration}).
\end{proof}

\begin{example} The following example illustrates Corollary~\ref{cor:elaborate}.

\noindent\begin{minipage}{0.28\textwidth}
\begin{figure}[H]
\begin{center}
\begin{tikzpicture}[xscale=1, yscale=1]
\draw (0,0)--(-1,-1);
\draw (-1,-1)--(-1.5,-2);
\draw (-1,-1)--(-0.5,-2);
\draw (0,0)--(1,-1);
\node at (0,0) [above] {\small{$0$}};
\node at (-1,-1) [left] {\small{$1$}};
\node at (1,-1) [right] {\small{$2$}};
\node at (-1.5,-2) [left] {\small{$3$}};
\node at (-0.5,-2) [right] {\small{$4$}};
\node at (-0.5,-0.5) [above left] {\small{$t_1$}};
\node at (0.5,-0.5) [above right] {\small{$t_2$}};
\node at (-1.75,-1.5) {\small{$t_3$}};
\node at (-0.25,-1.5) {\small{$t_4$}};
\end{tikzpicture}
\end{center}
%\caption{}
\label{fig:elaborate}
\end{figure}
\end{minipage}
\begin{minipage}{0.72\textwidth}
The weighting is well ordered if $\mathrm{max}\{t_1,t_2\}<\mathrm{min}\{t_3,t_4\}$. Let $\E$ be a set of positive density in $T_{q+1}$. By Corollary~\ref{cor:elaborate}, once $t_1,t_2,t_3,t_4$ are sufficiently large, one can find $\{v_0,v_1,v_2,v_3,v_4\}\subseteq \E$ such that
\begin{compactenum}[$(1)$]
\item $d(o,v_0)=d(o,v_1)=d(o,v_2)=d(o,v_3)=d(o,v_4)$;
\item $d(v_0,v_1)=2t_1$, $d(v_0,v_2)=2t_2$, $d(v_1,v_3)=2t_3$, $d(v_1,v_4)=2t_4$, $d(v_0,v_3)=2t_3$, and $d(v_0,v_4)=2t_4$.
\end{compactenum}
The distances $d(v_1,v_2)$, $d(v_2,v_3)$, $d(v_2,v_4)$ and $d(v_3,v_4)$ are unknown.
\end{minipage}
\end{example}

\section{Sets of positive density in trees of bounded degree}\label{sec:trees2}

In this section let $T_q^o$ denote the rooted tree with root $o$ such that every vertex has precisely $q$ children. Thus $T_q^o$ is the homogeneous tree $T_q$ with one branch pruned off at the root~$o$. Let $T \subseteq T_q^o$ be a subtree containing $o$, and with no leaves.  We will assume that the tree $T$ has positive Hausdorff dimension. This is equivalent to the statement that the fractal set $K = \partial T \subset [0,1]$ obtained by $q$-adic expansion along the infinite geodesics based at $o$ in $T$ has positive Hausdorff dimension (see \cite[\S1.2]{BP}).

Let $S_n(T) = \{ v \in V(T) \, | \, d(o,v) = n \}$ be the sphere of radius $n$ in $T$, centred at~$o$. In this section we will adopt the following notion of density: We say that $\E \subseteq T$ has positive lower density if 
$$
d_*(\E)=\liminf_{N \to \infty} \frac{1}{N}\sum_{n=1}^{N}\frac{|\E\cap S_n(T)|}{|S_n(T)|}>0.
$$
It is clear that if $d_*(\E)>0$ then $d^*(\E)>0$ too.

Theorem~\ref{thm:main1} raises the following natural question (we thank Itai Benjamini for asking us this question). For $t\in\NN$ and $\E\subseteq V$ let 
$
\E^t=\{x\in \E\mid \text{there exists $y\in\E$ with $d(x,y)=t$}\}.
$

\begin{question} Let $T\subseteq T_q^o$ be as above. Is it true that if $T$ has positive Hausdorff dimension, and if $\E \subseteq V(T)$ has positive lower density $d_*(\E)>0$, then there exists a subgroup $H$ of $\ZZ$ such that for sufficiently large $t\in H$ we have $\E \cap \E^t \neq \emptyset$? In other words, does there exist $k \geq 1$ and $K>0$ such that if $t \geq K$ then $\E \cap \E^{kt} \neq \emptyset$?
\end{question}

Note that Theorem~\ref{thm:main1} gives an affirmative answer in the case that $T=T_q^o$. However we will show below that generally the answer to the above question is negative. 

\begin{prop}
There exists  a tree $T \subset T_2^o$ of positive Hausdorff dimension, and a subset $\E \subset V(T)$ of positive lower density $d_*(\E)>0$ such that for any $k, K \ge 1$ there exists $t \ge K$ such that $\E \cap \E^{kt} = \emptyset$.
\end{prop}

\begin{proof}
For a subset $A\subseteq \ZZ$ we denote the \textit{density} of $A$ by 
$$
d(A)=\lim_{n\to \infty}\frac{|\{x\in A\mid -n\leq x\leq n\}|}{2n+1}
$$
provided the limit exists.

Let $A \subseteq \mathbb{Z}$ be a non-periodic Bohr set of density $d(A)$ with $d((A-A) + (A-A)) < 1$.  An explicit example is given by $A=\{n\in\ZZ\mid \{n\sqrt{2}\}<\epsilon\}$, where $\{x\}=x-\lfloor x\rfloor\geq 0$ denotes the fractional part of~$x\in\RR$, and $\epsilon<1/4$. Then $d(A)=\epsilon$ and $d((A-A)+(A-A))=4\epsilon$.

Let $T_A\subseteq T_2^o$ be the tree which has branching at each vertex of level $n$ for all $n \in A$. The Hausdorff dimension of $T_A$ is $\dim(T_A)=d(A)>0$ (see \cite[Example~3.3]{BP}). Now choose
 $$
 \E =\bigcup_{n \in A\,\cap\,\NN} S_n(T_A).
 $$
 It is clear that $d_*(\E)\geq d(A)$, and hence $d_*(\E)>0$. Moreover, if $v,v'\in \E$ then $d(v,v')\in (A-A)+(A-A)$. Since $A$ is non-periodic and $d((A-A)+(A-A)) < 1$ we see that $(A-A)+(A-A)$ is also a non-periodic Bohr set. Any non-periodic Bohr set $B \subseteq \mathbb{Z}$ satisfies the uniformity property along any infinite arithmetic progression:  for every $k \ge 1$ and any $m \in \{0,1,2,\ldots,k-1\}$ we have
\[
d(B) = \frac{d(B \cap (m + k\ZZ))}{k},
\]  
completing the proof.
 \end{proof}

\section{Sets of positive density in affine buildings}\label{sec:buildings}

In this section we extend the results of Section~\ref{sec:trees} to affine buildings of certain types (note that the tree $T_{q}$ is an affine building of type $\tilde{A}_1$). Sections~\ref{subsec:1}, \ref{subsec:2} and~\ref{subsec:3} recall the required background from the theory of affine buildings, mainly following the setup from~\cite{Par:06a,Par:06b} (see \cite{AB:08} for a comprehensive reference to building theory). In Section~\ref{subsec:5a} we define sets of positive density, and state our main theorem on sets of positive density in affine buildings. Section~\ref{subsec:4a} develops the theory of atoms required to give the proof of the main theorem in Section~\ref{subsec:6a}. The importance of affine buildings stems from their appearance in the theory of $p$-adic Lie groups, where they play an analogous role to the symmetric space for real Lie groups (see \cite{BT:72}; however in dimensions~$1$ and $2$ we note that not all affine buildings are associated to such a group). Thus we conclude in Section~\ref{subsec:6} with an application of our results to sets of positive density in $p$-adic Lie groups.

\subsection{Affine Coxeter groups and the Coxeter complex}\label{subsec:1}

Recall that a \textit{Coxeter system} $(W,S)$ is a group $W$ generated by a finite set $S$ with relations $(st)^{m_{st}}=1$ for all $s,t\in S$, where $m_{ss}=1$ for all $s\in S$, and $m_{st}=m_{ts}\in\ZZ_{\geq 2}\cup\{\infty\}$ for all $s\neq t$ (if $m_{st}=\infty$ then it is understood that there is no relation between $s$ and $t$). The \textit{length} of $w\in W$ is 
$$
\ell(w)=\min\{k\geq 0\mid w=s_1\cdots s_k\,\,\text{with $s_1,\ldots,s_k\in S$}\},
$$
and an expression $w=s_1\cdots s_k$ with $k$ minimal (that is, $k=\ell(w)$) is called a \textit{reduced expression} for $w$. A Coxeter system $(W,S)$ is \textit{irreducible} if $S$ cannot be partitioned into nonempty sets $S_1$ and $S_2$ with $st=ts$ for all $s\in S_1$ and $t\in S_2$, \textit{spherical} if $|W|<\infty$, and \textit{affine} if there exists a normal abelian subgroup $Q\triangleleft W$ of finite index.

All irreducible affine Coxeter systems arise in the following concrete way. Let $\Phi$ be a reduced, irreducible, crystallographic, finite root system in an $n$-dimensional real inner product space $E$ (see \cite[Chapter~VI]{bourbaki}). The \textit{dual root system} is $\Phi^{\vee}=\{\alpha^{\vee}\mid \alpha\in \Phi\}$, where $\alpha^{\vee}=2\alpha/\langle\alpha,\alpha\rangle$. Let $\{\alpha_1,\ldots,\alpha_n\}\subseteq \Phi$ be a choice of simple roots of $\Phi$, and let $\Phi^+\subseteq \Phi$ be the associated set of positive roots. The root system $\Phi$ has a unique \textit{highest root} $\varphi=m_1\alpha_1+\cdots+m_n\alpha_n$ (the \textit{height} of a root $\alpha=a_1\alpha_1+\cdots+a_n\alpha_n$ is $a_1+\cdots+a_n$).

The \textit{Weyl group} of $\Phi$ is the finite subgroup $W_0$ of $GL(E)$ generated by the orthogonal reflections $s_{\alpha}$ in the hyperplanes $H_{\alpha}=\{x\in E\mid \langle x,\alpha\rangle=0\}$ for $\alpha\in \Phi$. Let $s_i=s_{\alpha_i}$ for $1\leq i\leq n$, and let $S_0=\{s_i\mid 1\leq i\leq n\}$. Then $(W_0,S_0)$ is an irreducible spherical Coxeter system. Let $w_0\in W_0$ be the \textit{longest element} of $W_0$ (the unique element of maximal length). 

For each $\alpha\in \Phi$ and each $k\in\ZZ$ let $H_{\alpha,k}=\{x\in E\mid \langle x,\alpha\rangle=k\}$ (thus the affine hyperplane $H_{\alpha,k}$ is a translate of the linear hyperplane $H_{\alpha}=H_{\alpha,0}$). Let $s_{\alpha,k}\in\mathrm{Aff}(E)$ be the affine orthogonal reflection in $H_{\alpha,k}$, given by $s_{\alpha,k}(x)=x-(\langle x,\alpha\rangle-k)\alpha^{\vee}$. The \textit{affine Weyl group} of $\Phi$ is the subgroup of $\mathrm{Aff}(E)$ generated by the reflections $s_{\alpha,k}$ with $\alpha\in\Phi$ and $k\in\ZZ$. 

Writing $s_0=s_{\varphi,1}$ (with $\varphi$ the highest root) and $S=S_0\cup\{s_0\}$, the pair $(W,S)$ is a Coxeter system. Moreover, 
$$
W=Q\rtimes W_0\quad\text{where $Q=\ZZ\alpha_1^{\vee}+\cdots+\ZZ\alpha_n^{\vee}\cong \ZZ^n$ is the \textit{coroot lattice}},
$$
where we identify $\la\in Q$ with the translation $t_{\la}\in\mathrm{Aff}(E)$ given by $t_{\la}(x)=x+\la$. Thus $(W,S)$ is an affine Coxeter system, and all irreducible affine Coxeter systems arise in this way. In the standard Lie theory nomenclature $\Phi$ has a ``type'' $X_n$, where $X\in\{A,B,C,D,E,F,G\}$, and we say that $(W,S)$ has ``type'' $\tilde{X}_n$, and that $(W,S)$ has \textit{dimension}~$n$.

The \textit{fundamental coweights} are the dual basis $\omega_1,\ldots,\omega_n\in E$ to $\alpha_1,\ldots,\alpha_n$, given by $\langle\omega_i,\alpha_j\rangle=\delta_{i,j}$. The \textit{coweight lattice} of $\Phi$ is 
$$
P=\{\lambda\in E\mid \langle \la,\alpha\rangle\in\ZZ\text{ for all $\alpha\in\Phi$}\}=\ZZ\omega_1+\cdots+\ZZ\omega_n.
$$
Note that $Q\subseteq P$ (because $\langle\alpha,\beta^{\vee}\rangle\in\ZZ$ for all $\alpha,\beta\in\Phi$ by the crystallographic condition). Let $\rho=\omega_1+\cdots+\omega_n$. The set of \textit{dominant} coweights and \textit{strongly dominant} coweights are, respectively,
$$
P^+=\ZZ_{\geq 0}\omega_1+\cdots+\ZZ_{\geq0}\omega_n\quad\text{and}\quad P^{++}=\ZZ_{>0}\omega_1+\cdots+\ZZ_{>0}\omega_n=\rho+P^+.
$$
There is a natural partial order on $P^+$ given by $\mu\leq\la$ if and only if $\la-\mu\in P^+$. We write $\mu\ll \la$ if and only if $\la-\mu\in P^{++}$. Thus $\mu\ll\la$ if and only if $\la-\mu\geq \rho$.% (note that $\mu< \la$ is not equivalent to $\mu\leq\la$ and $\mu\neq \la$, and so the notation here is slightly nonstandard).

The family of hyperplanes $H_{\alpha,k}$, $\alpha\in \Phi$, $k\in\ZZ$, tessellates $E$ into $n$-dimensional geometric simplices, called \textit{chambers} (in the literature these are also called \textit{alcoves}). The \textit{fundamental chamber} is 
$$
\fc_0=\{x\in E\mid \langle x,\alpha_i\rangle\geq 0\text{ for $1\leq i\leq n$, and }\langle x,\varphi\rangle\leq 1\}.
$$
The group $W$ acts simply transitively on the set of chambers, and we often identify $W$ with the set of chambers via $w\leftrightarrow w\fc_0$. The extreme points of the chambers are \textit{vertices}, and each chamber has exactly $n+1$ vertices. The resulting simplicial complex $\Sigma(W,S)$ is called the \textit{Coxeter complex} of the affine Coxeter system. %The \textit{relative position} (or \textit{Weyl distance}) between chambers $u,v\in W$ is
%$\delta(u,v)=u^{-1}v$. Note that this ``distance'' takes values in $W$. 

Each vertex $x$ of $\Sigma(W,S)$ can be assigned a \textit{type} $\tau(x)\in\{0,1,\ldots,n\}$ as follows.  The fundamental chamber $\fc_0$ has vertices $x_0,x_1,\ldots,x_n$, where $x_0=0$ and $x_i=\omega_i/m_i$ for $1\leq i\leq n$ (with $\omega_i$ and $m_i$ as above), and we declare $\tau(x_i)=i$ for $0\leq i\leq n$. This extends uniquely to all vertices of $\Sigma(W,S)$ by requiring that every chamber has precisely one vertex of each type. The type of a simplex $\sigma$ is $\tau(\sigma)=\{\tau(x)\mid \text{$x$ is a vertex of $\sigma$}\}$, and the \textit{cotype} of $\sigma$ is $\{0,1,\ldots,n\}\backslash\tau(\sigma)$. The action of $W$ on $\Sigma(W,S)$ is type preserving. Both $Q$ and $P$ are subsets of the vertex set of $\Sigma(W,S)$. Specifically, $Q$ is the set of all type $0$ vertices, and $P$ is the set of all vertices $x$ with $\tau(x)\in \{0\}\cup\{1\leq i\leq n\mid m_i=1\}$. Equivalently, $P$ is the set of vertices $x$ of $\Sigma(W,S)$ whose stabiliser in $W$ is isomorphic to $W_0$.

\noindent\begin{minipage}{0.5\textwidth}
\begin{figure}[H]
\begin{center}
\begin{tikzpicture}[scale=0.8]
\path [fill=lightgray!70] (0,0) -- (4.25,0) -- (4.25,4.25) -- (0,0);
\path [fill=gray!90] (0,0) -- (1,0) -- (1,1) -- (0,0);
\draw (-4.25,-2) -- (4.25,-2); 
\draw (-4.25,-1) -- (4.25,-1);
\draw (-4.25,0) -- (4.25,0);
\draw (-4.25,1) -- (4.25,1);
\draw (-4.25,2) -- (4.25,2);
\draw (-4.25,3) -- (4.25,3); 
\draw (-4.25,4) -- (4.25,4);
\draw (-4.25,-3) -- (4.25,-3);
\draw (-4.25,-4) -- (4.25,-4);
\draw (-2,-4.25) -- (-2,4.25);
\draw (-1,-4.25) -- (-1,4.25);
\draw (0,-4.25) -- (0,4.25);
\draw (1,-4.25) -- (1,4.25);
\draw (2,-4.25) -- (2,4.25);
\draw (-4,-4.25) -- (-4,4.25);
\draw (-3,-4.25) -- (-3,4.25);
\draw (3,-4.25) -- (3,4.25);
\draw (4,-4.25) -- (4,4.25);
\draw (-4.25,3.75)--(-3.75,4.25);
\draw (-4.25,1.75)--(-1.75,4.25);
\draw (-4.25,-0.25) -- (0.25,4.25);
\draw (-4.25,-2.25) -- (2.25,4.25);
\draw (-4.25,-4.25) -- (4.25,4.25);
\draw (-2.25,-4.25) -- (4.25,2.25);
\draw (-0.25,-4.25) -- (4.25,0.25);
\draw (1.75,-4.25)--(4.25,-1.75);
\draw (3.75,-4.25)--(4.25,-3.75);
\draw (4.25,3.75)--(3.75,4.25);
\draw (4.25,1.75)--(1.75,4.25);
\draw (4.25,-0.25) -- (-0.25,4.25);
\draw (4.25,-2.25) -- (-2.25,4.25);
\draw (4.25,-4.25) -- (-4.25,4.25);
\draw (2.25,-4.25) -- (-4.25,2.25);
\draw (0.25,-4.25) -- (-4.25,0.25);
\draw (-1.75,-4.25)--(-4.25,-1.75);
\draw (-3.75,-4.25)--(-4.25,-3.75);
\node at (0,0) {$\bullet$};
\node at (-2,0) {$\bullet$};
\node at (2,0) {$\bullet$};
\node at (-2,-2) {$\bullet$};
\node at (0,-2) {$\bullet$};
\node at (2,-2) {$\bullet$};
\node at (0,2) {$\bullet$};
\node at (-2,2) {$\bullet$};
\node at (2,2) {$\bullet$};
\node at (-4,-4) {$\bullet$};
\node at (-4,-2) {$\bullet$};
\node at (-4,0) {$\bullet$};
\node at (-4,2) {$\bullet$};
\node at (-4,4) {$\bullet$};
\node at (-2,-4) {$\bullet$};
\node at (-2,4) {$\bullet$};
\node at (0,-4) {$\bullet$};
\node at (0,4) {$\bullet$};
\node at (2,-4) {$\bullet$};
\node at (2,4) {$\bullet$};
\node at (4,-4) {$\bullet$};
\node at (4,-2) {$\bullet$};
\node at (4,0) {$\bullet$};
\node at (4,2) {$\bullet$};
\node at (4,4) {$\bullet$};
\node at (1,1) {\scriptsize{$\Box$}};
\node at (1,3) {\scriptsize{$\Box$}};
\node at (1,-1) {\scriptsize{$\Box$}};
\node at (1,-3) {\scriptsize{$\Box$}};
\node at (-1,1) {\scriptsize{$\Box$}};
\node at (-1,3) {\scriptsize{$\Box$}};
\node at (-1,-1) {\scriptsize{$\Box$}};
\node at (-1,-3) {\scriptsize{$\Box$}};
\node at (-3,1) {\scriptsize{$\Box$}};
\node at (-3,3) {\scriptsize{$\Box$}};
\node at (-3,-1) {\scriptsize{$\Box$}};
\node at (-3,-3) {\scriptsize{$\Box$}};
\node at (3,1) {\scriptsize{$\Box$}};
\node at (3,3) {\scriptsize{$\Box$}};
\node at (3,-1) {\scriptsize{$\Box$}};
\node at (3,-3) {\scriptsize{$\Box$}};
%\node [color=white] at (-3,1) {$\bullet$};
%\node at (-3,1) {$\circ$};
%\node [color=white] at (-3,3) {$\bullet$};
%\node at (-3,3) {$\circ$};
%\node [color=white] at (-3,1) {$\bullet$};
%\node at (-3,1) {$\circ$};
%\node [color=white] at (-3,-1) {$\bullet$};
%\node at (-3,-1) {$\circ$};
%\node [color=white] at (-3,-3) {$\bullet$};
%\node at (-3,-3) {$\circ$};
%
%\node [color=white] at (-1,1) {$\bullet$};
%\node at (-1,1) {$\circ$};
%\node [color=white] at (-1,3) {$\bullet$};
%\node at (-1,3) {$\circ$};
%\node [color=white] at (-1,1) {$\bullet$};
%\node at (-1,1) {$\circ$};
%\node [color=white] at (-1,-1) {$\bullet$};
%\node at (-1,-1) {$\circ$};
%\node [color=white] at (-1,-3) {$\bullet$};
%\node at (-1,-3) {$\circ$};
%
%\node [color=white] at (1,1) {$\bullet$};
%\node at (1,1) {$\circ$};
%\node [color=white] at (1,3) {$\bullet$};
%\node at (1,3) {$\circ$};
%\node [color=white] at (1,1) {$\bullet$};
%\node at (1,1) {$\circ$};
%\node [color=white] at (1,-1) {$\bullet$};
%\node at (1,-1) {$\circ$};
%\node [color=white] at (1,-3) {$\bullet$};
%\node at (1,-3) {$\circ$};
%
%\node [color=white] at (3,1) {$\bullet$};
%\node at (3,1) {$\circ$};
%\node [color=white] at (3,3) {$\bullet$};
%\node at (3,3) {$\circ$};
%\node [color=white] at (3,1) {$\bullet$};
%\node at (3,1) {$\circ$};
%\node [color=white] at (3,-1) {$\bullet$};
%\node at (3,-1) {$\circ$};
%\node [color=white] at (3,-3) {$\bullet$};
%\node at (3,-3) {$\circ$};
%
\node at (2.7,-1.7) {\small{$\alpha_1^{\vee}$}};
\node at (0.7,2.3) {\small{$\alpha_2^{\vee}$}};
\node at (2.7,0.2) {\small{$\omega_1$}};
\node at (1.6,1.15) {\small{$\omega_2$}};
\node at (3.2,0.9) [above right] {\small{$\rho$}};
\draw [latex-latex,line width=1pt] (0,-2)--(0,2);
\draw [latex-latex,line width=1pt] (-2,0)--(2,0);
\draw [latex-latex,line width=1pt] (-2,-2)--(2,2);
\draw [latex-latex,line width=1pt] (-2,2)--(2,-2);
\draw [-latex, line width=1pt] (0,0)--(1,1);
\node at (-2.1,3.8) [below right] {\small{$x$}};
\node at (2.2,-4.1) [above right] {\small{$y$}};
\end{tikzpicture}
\caption{The Coxeter complex of type $\tilde{C}_2$}
\label{fig:rootsystem}
\end{center}
\end{figure}
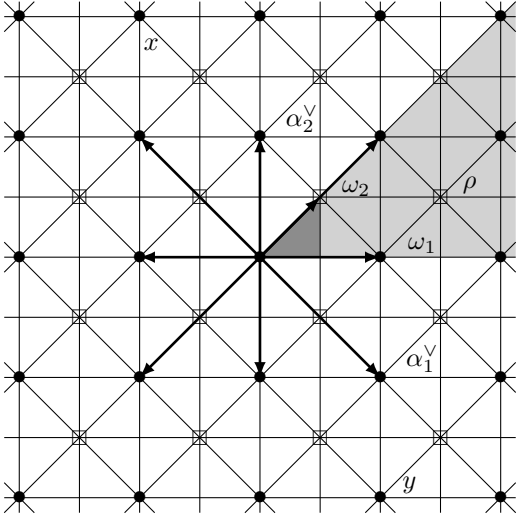
\end{minipage}
\begin{minipage}{0.5\textwidth}
The root system of type $C_2$ is 
$$\Phi=\pm\{\alpha_1,\alpha_2,\alpha_1+\alpha_2,2\alpha_1+\alpha_2\},$$ where $\alpha_1=e_1-e_2$ and $\alpha_2=2e_2$. We have $\alpha_1^{\vee}=\alpha_1$ and $\alpha_2^{\vee}=e_2$, and the dual root system is $\Phi^{\vee}=\pm\{\alpha_1^{\vee},\alpha_2^{\vee},\alpha_1^{\vee}+\alpha_2^{\vee},\alpha_1^{\vee}+2\alpha_2^{\vee}\}$. The fundamental coweights are $\omega_1=e_1$ and $\omega_2=\frac{1}{2}e_1+\frac{1}{2}e_2$. The coroot lattice $Q$ is the set of $\bullet$ vertices, and the coweight lattice~$P$ is the union of the $\bullet$ and {\scriptsize$\Box$} vertices. The fundamental chamber is darkly shaded, and the cone of dominant coweights is lightly shaded. The points $x,y$ are marked for later reference. 
\end{minipage}

\subsection{Affine buildings}\label{subsec:2}

Let $(W,S)$ be an affine Coxeter system of dimension~$n$, as constructed in the previous section. Let $\Delta$ be an affine building of type $(W,S)$. Let us briefly expand on this (see \cite{AB:08}). Thus $\Delta$ is a very special kind of simplicial complex, whose maximal simplices are called \textit{chambers}, and all chambers of $\Delta$ have dimension~$n$. Moreover, $\Delta$ is equipped with a distinguished collection of sub-simplicial complexes, called \textit{apartments}, satisfying three axioms:\smallskip
\begin{compactenum}
\item[(B1)] all apartments are isomorphic to the Coxeter complex $\Sigma(W,S)$;
\item[(B2)] if $c,d\in\Delta$ are chambers of $\Delta$, then there exists an apartment $A$ containing both of them;
\item[(B3)] if $A,A'$ are apartments containing a common chamber, then there exists a unique simplicial complex isomorphism $\psi:A\to A'$ fixing every simplex of $A\cap A'$. 
\end{compactenum} \smallskip

Thus one may regard $\Delta$ as being made by ``gluing together'' many copies of $\Sigma(W,S)$. Axiom~(B2) tells us that when determining the relative position between two simplices we can make the measurement in an apartment, and the content of Axiom~(B3) is that the measurement we obtain is independent of the particular apartment chosen.
\smallskip

\noindent\begin{minipage}{0.7\textwidth}
\hspace{0.7cm}A \textit{panel} is a codimension~$1$ simplex $\pi$ of $\Delta$. Chambers $c,d$ of $\Delta$ are called \textit{adjacent} (written $c\sim d$) if $c\cap d$ is a panel. The figure illustrates the local picture in the $n=2$ case. Here each chamber has $3$ vertices, and panels are edges. The $5$ chambers shown all share a common panel, and hence are mutually adjacent. We say that $\Delta$ has \textit{uniform thickness parameter} $q$ if $|\{c\mid c\sim d\}|=q$ for all chambers $d$ of $\Delta$. The figure illustrates the local picture in the case $q=4$.
\end{minipage}
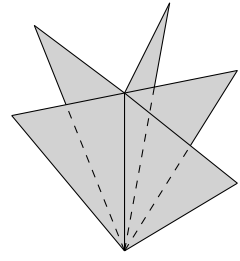
\begin{minipage}{0.3\textwidth}
\begin{center}
\begin{tikzpicture}[scale=0.3]
\path [fill=lightgray!70] (0,5)--(-5,4) -- (0,-2) -- (0,5);
\path [fill=lightgray!70] (0,-2)--(5,1) -- (0,5) -- (0,-2);
\path [fill=lightgray!70] (2.917, 2.666)-- (5, 6)--(0,5)--(2.917, 2.666);
\path [fill=lightgray!70] (1.321, 5.264)-- (2, 9)--(0,5)--(1.321, 5.264);
\path [fill=lightgray!70] (-2.5926, 4.4815)-- (-4, 8)--(0,5)--(-2.5926, 4.4815);
    \draw (-5, 4)--( 0, -2);
   \draw (0, 5)-- (-5, 4);
    \draw (0 ,-2) --(0, 5);
    \draw (0, -2)-- ( 5, 1);
    \draw (5, 1) --(0, 5);
    \draw (0, 5)-- (2, 9);
   \draw (2.917, 2.666)-- (5, 6);
    \draw  (2, 9) -- (1.321, 5.264);
    \draw ( 0, 5)--  (5, 6);
    \draw (-2.5926, 4.4815) --( -4, 8);
    \draw  (-4, 8) -- (0, 5);
     \draw[style=dashed]  (0, -2)--  (2.917, 2.666 );
    \draw[style=dashed] (1.321, 5.264)-- ( 0, -2);
    \draw[style=dashed] (0, -2)-- (-2.5926, 4.4815);
    %\node at (3,1) {$x$};
    %\node at (5.75,1) {$v$};
    %\node at (5,1) {$\bullet$};
\end{tikzpicture}
\end{center}
\end{minipage}

\begin{example}
The $A_1$ root system is $\Phi=\{-\alpha,\alpha\}$ in the $1$-dimensional space $E=\RR\alpha$. The Coxeter complex of the associated affine Coxeter system is a tessellation of $\RR$ by intervals. It is thus clear from the axioms above that $\tilde{A}_1$ buildings are simply trees in which every vertex has valency at least~$2$ (that is, there are no ``leaves''). The chambers are the edges, and the panels are the vertices. In this case the building is easy to draw, however in higher dimension the ``thickness'' of the building is difficult to visualise, and so our pictures are typically of a piece of an apartment of $\Delta$, with the branching left to the reader's imagination.
\end{example}

Fix, once and for all, an apartment $A_0$ of $\Delta$, and an isomorphism $\psi_0:A_0\to\Sigma(W,S)$, and identify $A_0$ with $\Sigma(W,S)$ via~$\psi_0$. Thus we regard $\Sigma(W,S)\equiv A_0$ as an apartment of $\Delta$ (the \textit{base apartment}). We write $o=\psi_0^{-1}(0)$ (the \textit{root} of $\Delta$). The type map $\tau$ on $\Sigma(W,S)$ extends uniquely to all vertices of $\Delta$ by requiring that every chamber has precisely one vertex of each type.

For each $0\leq i\leq n$ we define an adjacency relation $\sim_i$ on chambers of $\Delta$ by setting $c\sim_i d$ if and only if $c\cap d$ is a panel of cotype~$i$ (that is, $c$ and $d$ share all vertices except for their type $i$ vertices). Then $c\sim d$ if and only if $c\sim_i d$ for some $0\leq i\leq n$. The \textit{relative position} $\delta(c,d)\in W$ between chambers $c,d$ of $\Delta$ is defined by choosing a path
\begin{align}\label{eq:gallery}
c=c_0\sim_{i_1}c_1\sim_{i_2}\cdots\sim_{i_k}c_k=d
\end{align}
of minimal length joining $c$ to $d$, and setting 
$
\delta(c,d)=s_{i_1}\cdots s_{i_k}.
$
The building axioms ensure that the value of $\delta(c,d)\in W$ is independent of the particular choice of minimal path made in~(\ref{eq:gallery}). The \textit{numerical distance} $\mathrm{dist}(c,d)$ between $c$ and $d$ is the length of a minimal length gallery from $c$ to $d$. Thus $\mathrm{dist}(c,d)=\ell(\delta(c,d))$. A sequence of chambers as in~(\ref{eq:gallery}) is called a \textit{gallery} of \textit{type} $(i_1,\ldots,i_k)$ in the building theory vernacular. A basic fact is that a gallery of type $(i_1,\ldots,i_k)$ joining $c$ to $d$ has minimal length amongst all galleries joining $c$ to $d$ if and only if $s_{i_1}\cdots s_{i_k}$ is a reduced expression (that is, $\ell(s_{i_1}\cdots s_{i_k})=k$).

Let $V$ denote the set of all vertices of $\Delta$. Let $V_Q=\{x\in V\mid \tau(x)=0\}$, and let 
$$
V_P=\{x\in V\mid \tau(x)\in \tau(P)\}.
$$
Then $V_Q\subseteq V_P$. The elements of $V_P$ are called the \textit{special vertices} of $\Delta$. For all affine buildings other than those of type $\tilde{A}_n$ the set $V_P$ is a strict subset of $V$. %The special vertices are characterised as those vertices $x\in V$ with the property that the maximal numerical distance between chambers $c$ and $d$ of $\Delta$ with $x\in c$ and $x\in d$ is $\ell(w_0)$. If $x$ is a non-special vertex then this maximal numerical distance is strictly smaller than~$\ell(w_0)$.

\subsection{Vector distance and spheres}\label{subsec:3}

Henceforth we let $(W,S)$ be an irreducible affine Coxeter system of dimension~$n$, and let $\Delta$ be an affine building of type $(W,S)$ with uniform thickness parameter $2\leq q<\infty$. 

Let $x,y\in V_P$ be special vertices of $\Delta$. The \textit{vector distance} $\vect{d}(x,y)\in P^+$ from $x$ to $y$ is defined as follows. Choose an apartment $A$ containing $x$ and $y$ (using (B2)), and let $\psi:A\to\Sigma(W,S)$ be a type preserving simplicial complex isomorphism (using (B1)). Define
$$
\vect{d}(x,y)=(\psi(y)-\psi(x))^+,
$$
where for $\mu\in P$ we write $\mu^+$ for the unique element in $W_0\mu\cap P^+$. This value is independent of the choice of apartment $A$ and isomorphism $\psi$ (using (B3); see \cite[Proposition~5.6]{Par:06a}). Somewhat more informally, to compute $\vect{d}(x,y)$ one looks at the vector from $x$ to $y$ (in any apartment containing $x$ and $y$) and takes the dominant representative of this vector under the $W_0$-action. For example, if $x,y\in V_P$ lie in an apartment as illustrated in Figure~\ref{fig:rootsystem} then $\vect{d}(x,y)=2\omega_1+4\omega_2$.

We have (\cite[Proposition~5.8]{Par:06a})
$$
\vd(y,x)=\vd(x,y)^*\quad\text{for all $x,y\in V_P$},
$$
where $\la^*=-w_0\la$ (with $w_0$ the longest element of $W_0$). We say that $\Delta$ is of $(-1)$-type if $w_0$ acts on $E$ by $-1$. Thus $\Delta$ has $(-1)$-type if and only if $\vect{d}(y,x)=\vect{d}(x,y)$ for all $x,y\in V_P$. By direct examination of root systems, the irreducible $(-1)$-type affine buildings are those of types $\tilde{A}_1$, $\tilde{B}_n$ ($n\geq 2$), $\tilde{C}_n$ ($n\geq 2$), $\tilde{D}_n$ ($n\geq 4$ even), $\tilde{E}_7$, $\tilde{E}_8$, $\tilde{F}_4$, and $\tilde{G}_2$. In other words, the affine buildings that are not of $(-1)$-type are those of types $\tilde{A}_n$ ($n\geq 2$), $\tilde{D}_n$ ($n\geq 4$ odd), and $\tilde{E}_6$.

For $\la\in P^+$ and $x\in V_P$ the sphere of radius $\lambda$ and centre $x$ is 
$$
S_{\la}(x)=\{y\in V_P\mid \vect{d}(x,y)=\la\}.
$$
We write $S_{\la}=S_{\la}(o)$. The cardinality $|S_{\la}(x)|$ does not depend on $x\in V_P$. In fact, by \cite[Proposition~1.5]{Par:06b} we have
\begin{align}\label{eq:sphere}
|S_{\la}(x)|=\frac{W_0(q^{-1})}{W_{0\la}(q^{-1})}q^{\langle\la,2\rho\rangle},
\end{align}
where $W_{0\la}=\{w\in W_0\mid w\la=\la\}$ and for finite subsets $U\subseteq W$ we write $U(q^{-1})=\sum_{u\in U}q^{-\ell(u)}$.

\begin{cor}\label{cor:sphere}
Suppose that $\mu\in P^{++}$ and that $\mu\leq \la$. Then 
$$
|S_{\la}|=|S_{\mu}|q^{\langle\la-\mu,2\rho\rangle}.
$$
\end{cor}

\begin{proof}
If $\la,\mu\in P^{++}$ then $W_{0\la}=W_{0\mu}=\{1\}$, and the result follows from~(\ref{eq:sphere}). 
\end{proof}

\subsection{Sets of positive density}\label{subsec:5a}

Recall that $(W,S)$ is an irreducible affine Coxeter system of dimension~$n$, and $\Delta$ is an affine building of type $(W,S)$ with uniform thickness parameter $2\leq q<\infty$. 

The \textit{(upper) density} of a subset $\E\subseteq V_P$ is 
\begin{align}\label{eq:defdensity}
d^*(\E)=\limsup_{\la\to\infty}\frac{|\E\cap S_{\la}|}{|S_{\la}|}=\lim_{\la\to\infty}\left(\sup_{\mu\geq \la}\frac{|\E\cap S_{\mu}|}{|S_{\mu}|}\right),
\end{align}
where the limit is taken with each $\langle\la,\alpha_j\rangle$ tending to~$\infty$. We note that~(\ref{eq:defdensity}) is well defined, because writing $a_{\la}=|\E\cap S_{\la}|/|S_{\la}|$ and $b_{\la}=\sup\{a_{\mu}\mid \mu\geq \la\}$ we have that $b_{\mu}\geq b_{\la}$ whenever $\mu\leq \la$. Writing 
$$
\underline{\la}=\min\{\langle\la,\alpha_j\rangle\mid 1\leq j\leq n\}\rho\quad\text{and}\quad \overline{\la}=\max\{\langle\la,\alpha_j\rangle\mid 1\leq j\leq n\}\rho,
$$
we have $\underline{\la}\leq\la\leq\overline{\la}$ and so $b_{\overline{\la}}\leq b_{\la}\leq b_{\underline{\la}}$ for all $\la\in P^{+}$. Since $\lim_{m\to\infty}b_{m\rho}=L$ exists (by monotone convergence) we have $b_{\la}\to L$ whenever $\la\to\infty$ with each $\langle\la,\alpha_j\rangle\to\infty$. Moreover, as in the tree case, the property of having positive density is easily seen to be independent of the choice of root vertex~$o$ (however the numerical value of $d^*(\E)$ depends on the choice of root).

Recall that we write $\mu\ll\la$ if and only if $\la-\mu\in P^{++}$. For each $k>0$ let 
$$
\mathbb{M}^k=\{(\la_1,\la_2,\ldots,\la_k)\in P^k\mid 0\ll\la_1\ll\la_2\ll\cdots\ll\la_k\}.
$$
Let $\vect{\la}=(\la_i)_{i=1}^k\in\mathbb{M}^k$ and let $\vect{r}=(r_i)_{i=1}^k\in\mathbb{N}^k$. Then by a \textit{$(k,\vect{\la},\vect{r})$-star} we mean a set $Y$ of special vertices of $\Delta$ with a distinguished vertex $v_0\in Y$ (called the \textit{centre} of $Y$) such that
\begin{compactenum}[$(1)$]
\item if $v\in Y$ then $\vect{d}(v_0,v)\in\{2\la_1,\ldots,2\la_k\}$,
\item for each $1\leq i\leq k$ we have $|\{v\in Y\mid \vect{d}(v_0,v)=2\la_i\}|=r_i$.
\end{compactenum}
We call $Y$ \textit{balanced} if $\vect{d}(o,x)$ is constant for all $x\in Y$ (that is, $Y\subseteq S_{\mu}$ for some $\mu\in P^+$). %We note that balanced stars are not as rigid as in Figure~\ref{fig:configuration}, in the sense that one can not deduce the distances $\vect{d}(x,y)$ between vertices $x

The main theorem of this section is the following analogue of Theorem~\ref{thm:main1}. 

\begin{thm}\label{thm:main1buildings}
Suppose that $\Delta$ has $(-1)$-type. Let $\E\subseteq V_P$ with $d^*(\E)>0$. For each $k>0$ and each $\vect{r}\in \mathbb{N}^k$ there exists a constant $K=K(\E,k,\vect{r})>0$ such that $\E$ contains a balanced $(k,\vect{\la},\vect{r})$-star for all sequences $\vect{\la}=(\la_i)_{i=1}^k\in\mathbb{M}^k$ with $\la_1\geq K\rho$.
\end{thm}

Before proving Theorem~\ref{thm:main1buildings} we define projection maps and atoms in affine buildings (Section~\ref{subsec:4a}), and prove a series of preliminary results (Section~\ref{subsec:6a}).

\subsection{Projections and atoms}\label{subsec:4a}

By \cite[Corollary~B.3]{Par:06b}, if $\nu,\mu\in P^+$ with $\mu\leq \nu$, and if $y\in S_{\nu}$, then there is a unique vertex $x\in S_{\mu}$ such that $\vd(x,y)=\nu-\mu$. This allows us to define \textit{projection maps}: If $\mu\leq\nu$ let
$$
\pi_{\nu,\mu}:S_{\nu}\to S_{\mu},\quad\text{where $\pi_{\nu,\mu}(y)$ is the unique vertex $x\in S_{\mu}$ with $\vd(x,y)=\nu-\mu$}.
$$
Now, if $x\in S_{\mu}$ and $\la\in P^+$ then the set of ``$\la$-children'' (or $\la$-descendants) of~$x$ is the set $C(x,\la)$ of those $y\in S_{\mu+\la}$ that project back to~$x$. That is, 
$$
C(x,\la)=\{y\in S_{\la+\mu}\mid \pi_{\la+\mu,\mu}(y)=x\}=\{y\in V_P\mid \vect{d}(x,y)=\la\text{ and }\vect{d}(o,y)=\mu+\la\}.
$$
We write 
$
C(x)=\bigcup_{\la\in P^+}C(x,\la)
$
for the set of all descendants of $x$.

In the case that $\vd(o,x)=\mu\in P^{++}$ and $\la\in P^{++}$ we decompose $C(x,\la)$ further, as follows. First we note that if $z,z'\in V_P$ are any special vertices with $\vd(z,z')\in P^{++}$ then there is a unique chamber $c(z,z')$ such that every gallery 
$
z\in c_0\sim c_1\sim\cdots\sim c_k\ni z'
$
of minimal length subject to $z\in c_0$ and $z'\in c_k$ starts with $c_0=c(z,z')$, as illustrated in Figure~\ref{fig:start}.

\begin{figure}[H]
\centering
\begin{tikzpicture}[scale=1]
%
%\path [fill=lightgray!70] (0,0)--(-2,0)--(-4,-2)--(-2,-2)--(0,0);
%\draw (0,0)--(-2,0)--(-4,-2)--(-2,-2)--(0,0);
\draw (-4,-2)--(-8,-2)--(-10,-4)--(-6,-4)--(-4,-2);
%\draw (-1,-1)--(-5,-1);
\draw (-5,-3)--(-9,-3);
%\draw (-1,0)--(-1,-1);
%\draw (-2,0)--(-2,-2);
%\draw (-3,-1)--(-3,-3);
%\draw (-4,-1)--(-4,-3);
\draw (-5,-2)--(-5,-3);
\draw (-6,-2)--(-6,-4);
\draw (-7,-2)--(-7,-4);
\draw (-8,-2)--(-8,-4);
\draw (-9,-3)--(-9,-4);
%\draw (-2,0)--(-1,-1);
%\draw (-3,-1)--(-2,-2);
%\draw (-5,-1)--(-3,-3);
\draw (-6,-2)--(-5,-3);
\draw (-8,-2)--(-6,-4);
\draw (-9,-3)--(-8,-4);
\draw (-6,-2)--(-8,-4);
%\draw [-latex] (-3.25,-1.75)--(-3.75,-1.25);
%\draw [-latex] (-3.75,-1.25)--(-4.25,-1.25);
%\draw [-latex] (-4.25,-1.25)--(-4.75,-1.75);
%\draw [-latex] (-4.75,-1.75)--(-4.75,-2.25);
\node at (-10.1,-2.6) {$c(z,z')$};
%\node at (-3.35,-1.75) {$c_x$};
%\node at (0,0) [right] {$o$};
\node at (-4,-2) [right] {$z'$};
\node at (-10,-4) [left] {$z$};
%\node at (0,0) {$\bullet$};
\node at (-4,-2) {$\bullet$};
\node at (-10,-4) {$\bullet$};
\draw [-latex] (-10,-3) to [bend right=10] (-9.25,-3.75);
\end{tikzpicture}
\caption{The chamber $c(z,z')$}\label{fig:start}
\end{figure}
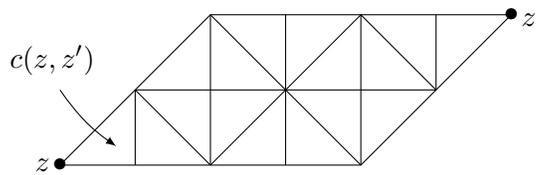

\noindent Let $c_x=c(x,o)$, and let $O(x)$ denote the set of all chambers $c$ of $\Delta$ with $x\in c$ such that $\mathrm{dist}(c_x,c)=\ell(w_0)$. These are the chambers ``opposite'' $c_x$ in the ``residue'' of $x$, and we have $|O(x)|=q^{\ell(w_0)}$. Then, for each $c\in O(x)$ define
\begin{align}\label{eq:decom}
C(x,c,\la)=\{y \in C(x,\la)\mid c(x,y)=c\}.
\end{align}

This situation is illustrated in Figure~\ref{fig:atomsbuilding} for $\tilde{C}_2$ buildings, where $\vect{d}(o,x)=\mu=\omega_1+2\omega_2$ and $\la=2\omega_1+2\omega_2$ (c.f. Figure~\ref{fig:rootsystem}, and note the convention, like in the tree, of drawing the building falling ``downwards'' from $o$). The grey shaded region is determined by $o$ and $x$. There are then $q^4$ choices for the chamber $c\in O(x)$ in the position shown. Then, for each such $c$, the set $C(x,c,\la)$ contains $q^{10}$ vertices (to make this count, choose a minimal length path from $c$ to position~$y$, and at each step there is thickness~$q$; see also Lemma~\ref{lem:atoms}(2) below). In the $1$-dimensional case of trees (that is, $\Phi=A_1$ and $\la\in\NN$), the shaded region is just the geodesic $[o,x]$ joining vertices $o$ and $x$, and the chambers $c$ are just the $q$ edges incident with $x$ and not contained in $[o,x]$. Thus $C(x,c,\la)=C(y,\la-1)$ in the tree case, where the edge $c$ has vertices $x,y$ (see~(\ref{eq:treedecomposition})).

\begin{figure}[H]
\centering
\begin{tikzpicture}[scale=1]
\path [fill=lightgray!70] (0,0)--(-2,0)--(-4,-2)--(-2,-2)--(0,0);
\draw (0,0)--(-2,0)--(-4,-2)--(-2,-2)--(0,0);
\draw (-4,-2)--(-8,-2)--(-10,-4)--(-6,-4)--(-4,-2);
\draw (-1,-1)--(-5,-1);
\draw (-3,-3)--(-9,-3);
\draw (-1,0)--(-1,-1);
\draw (-2,0)--(-2,-2);
\draw (-3,-1)--(-3,-3);
\draw (-4,-1)--(-4,-3);
\draw (-5,-1)--(-5,-3);
\draw (-6,-2)--(-6,-4);
\draw (-7,-2)--(-7,-4);
\draw (-8,-2)--(-8,-4);
\draw (-9,-3)--(-9,-4);
\draw (-2,0)--(-1,-1);
\draw (-3,-1)--(-2,-2);
\draw (-5,-1)--(-3,-3);
\draw (-6,-2)--(-5,-3);
\draw (-8,-2)--(-6,-4);
\draw (-9,-3)--(-8,-4);
\draw (-6,-2)--(-8,-4);
%\draw [-latex] (-3.25,-1.75)--(-3.75,-1.25);
%\draw [-latex] (-3.75,-1.25)--(-4.25,-1.25);
%\draw [-latex] (-4.25,-1.25)--(-4.75,-1.75);
%\draw [-latex] (-4.75,-1.75)--(-4.75,-2.25);
\node at (-4.75,-2.35) {$c$};
\node at (-3.35,-1.75) {$c_x$};
\node at (-9.35,-3.8) {$c_y$};
\node at (0,0) [right] {$o$};
\node at (-3.8,-1.95) [below right] {$x$};
\node at (-10,-4) [left] {$y$};
\node at (0,0) {$\bullet$};
\node at (-4,-2) {$\bullet$};
\node at (-10,-4) {$\bullet$};
\end{tikzpicture}
\caption{Decomposing into atoms}\label{fig:atomsbuilding}
\end{figure}
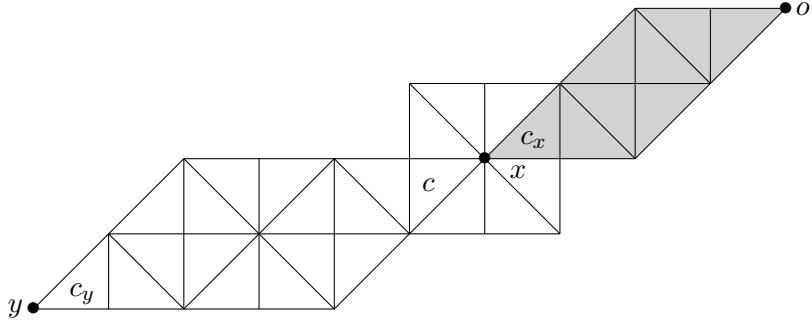

\begin{lemma}\label{lem:atoms}
Let $\la,\mu\in P^{++}$, and write $\nu=\la+\mu$. The members of 
$$
\cA_{\nu,\la}=\{C(x,c,\la)\mid x\in S_{\mu}\text{ and }c\in O(x)\}
$$
form a partition of the sphere $S_{\nu}$. Moreover we have
\begin{compactenum}[$(1)$]
\item $|\cA_{\nu,\la}|=|S_{\mu}|q^{\ell(w_0)}$, and
\item $|C(x,c,\la)|=q^{\langle \la,2\rho\rangle-\ell(w_0)}$, independent of $\mu\in P^{++}$, $x\in S_{\mu}$, and $c\in O(x)$.
\end{compactenum}
\end{lemma}

\begin{proof}
If $y\in S_{\nu}$ and $x\in S_{\mu}$ then $y\in C(x,\la)$ if and only if $x=\pi_{\nu,\mu}(y)$. Thus we have a disjoint union $S_{\nu}=\bigcup_{x\in S_{\mu}}C(x,\la)$. Moreover, if $y\in C(x,\la)$ with $x\in S_{\mu}$ then $c(x,y)\in O(x)$ (to see this, note that since $x=\pi_{\nu,\mu}(y)$ the three vertices $o,x,y$ lie in a common apartment, and hence are configured as in Figure~\ref{fig:atomsbuilding}, and in this figure $c(x,y)=c$). Then $y\in C(x,c,\la)$ if and only if $c(x,y)=c$, giving the disjoint union $C(x,\la)=\bigcup_{c\in O(x)}C(x,c,\la)$. 

Since $|O(x)|=q^{\ell(w_0)}$ for all $x\in S_{\mu}$, it is then clear that $|\cA_{\nu,\la}|=|S_{\mu}|q^{\ell(w_0)}$, and hence (1). To prove (2) we note that if $y\in C(x,c,\la)$ then $\mathrm{dist}(c_x,c_y)$ is equal to the number of hyperplanes of~$\Sigma(W,S)$ separating the chambers $\fc_0$ and $t_{\la}\fc_0$. Counting these hyperplanes by parallelism classes gives
$$
\mathrm{dist}(c_x,c_y)=\sum_{\alpha\in \Phi^+}\langle \la,\alpha^{\vee}\rangle=\langle\la,2\rho\rangle,
$$
where we use the fact that $2\rho=\sum_{\alpha\in\Phi^+}\alpha^{\vee}$ (see \cite[\S VI.10]{bourbaki}). Thus $\mathrm{dist}(c,c_y)=\mathrm{dist}(c_x,c_y)-\mathrm{dist}(c_x,c)=\langle\la,2\rho\rangle-\ell(w_0)$, and hence the result (see Figure~\ref{fig:atomsbuilding} for illustration).
\end{proof}

Let $\la,\mu\in P^{++}$ and write $\nu=\la+\mu$. Let $\cF_{\nu,\la}$ denote the $\sigma$-algebra generated by $\cA_{\nu,\la}$. The members of the set $\cA_{\nu,\la}$ are called the \textit{atoms} of $\cF_{\nu,\la}$.

%
%
%\begin{lemma}\label{lem:concatenation}
%If $\la,\mu\in P^{++}$ then $m_{\la}w_0m_{\mu}=m_{\la+\mu}$, and $\ell(m_{\la+\mu})=\ell(m_{\la})+\ell(w_0)+\ell(m_{\mu})$.  
%\end{lemma}
%
%
%\begin{proof}
%Since $\la\in P^{++}$ we have $t_{\la}=m_{\la}w_0$, and similarly for $t_{\mu}$. Thus
%$$
%m_{\la}w_0m_{\mu}=t_{\la}t_{\mu}w_0=t_{\la+\mu}w_0=m_{\la+\mu}.
%$$
%Moreover, since $m_{\la+\mu}=t_{\la+\mu}w_0$ we have $\ell(m_{\la+\mu})=\ell(t_{\la+\mu})-\ell(w_0)$, and since $\ell(t_{\la+\mu})=\langle\la+\mu,2\rho\rangle$ we have
%\begin{align*}
%\ell(m_{\la+\mu})&=\langle\la+\mu,2\rho\rangle-\ell(w_0)\\
%&=(\langle\la,2\rho\rangle-\ell(w_0))+(\langle\mu,2\rho\rangle-\ell(w_0))+\ell(w_0)\\
%&=\ell(m_{\la})+\ell(w_0)+\ell(m_{\mu}),
%\end{align*}
%completing the proof.
%\end{proof}

\newpage

\subsection{Proof of Theorem~\ref{thm:main1buildings}}\label{subsec:6a}

\begin{lemma}\label{lem:doubledistance}
Let $z\in V_P$, and suppose that $x_1,x_2\in V_P$ satisfy (for $j=1,2$):
\begin{compactenum}[$(1)$]
\item $\vect{d}(z,x_j)\in P^{++}$, 
\item $\vect{d}(o,x_j)=\vect{d}(o,z)+\vect{d}(z,x_j)$,
\item $\mathrm{dist}(c_1,c_2)=\ell(w_0)$ where $c_j=c(z,x_j)$.
\end{compactenum}
Then $\vect{d}(x_1,x_2)=\vect{d}(x_1,z)+\vect{d}(z,x_2)$. 
\end{lemma}

\begin{proof}
We will only sketch the proof. For $j=1,2$, the \textit{convex hull} $\mathrm{conv}\{z,x_j\}$ is intersection of all apartments containing both $z$ and $x_j$. Equivalently, since $\vect{d}(z,x_j)\in P^{++}$, the convex hull $\mathrm{conv}\{z,x_j\}$ is the union of all chambers of $\Delta$ lying on a minimal length gallery from $c_j=c(z,x_j)$ to $c_j'=c(x_j,z)$. These convex hulls are shaded in Figure~\ref{fig:foldback}.

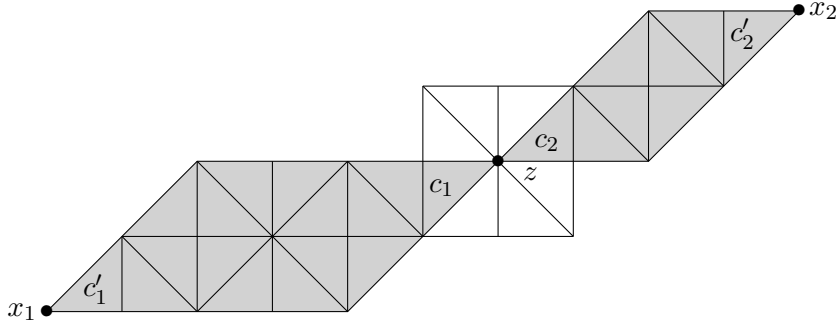
\begin{figure}[H]
\centering
\begin{tikzpicture}[scale=1]
\path [fill=lightgray!70] (0,0)--(-2,0)--(-4,-2)--(-2,-2)--(0,0);
\path [fill=lightgray!70] (-4,-2)--(-8,-2)--(-10,-4)--(-6,-4)--(-4,-2);
\draw (0,0)--(-2,0)--(-4,-2)--(-2,-2)--(0,0);
\draw (-4,-2)--(-8,-2)--(-10,-4)--(-6,-4)--(-4,-2);
\draw (-1,-1)--(-5,-1);
\draw (-3,-3)--(-9,-3);
\draw (-1,0)--(-1,-1);
\draw (-2,0)--(-2,-2);
\draw (-3,-1)--(-3,-3);
\draw (-4,-1)--(-4,-3);
\draw (-5,-1)--(-5,-3);
\draw (-6,-2)--(-6,-4);
\draw (-7,-2)--(-7,-4);
\draw (-8,-2)--(-8,-4);
\draw (-9,-3)--(-9,-4);
\draw (-2,0)--(-1,-1);
\draw (-3,-1)--(-2,-2);
\draw (-5,-1)--(-3,-3);
\draw (-6,-2)--(-5,-3);
\draw (-8,-2)--(-6,-4);
\draw (-9,-3)--(-8,-4);
\draw (-6,-2)--(-8,-4);
%\draw [-latex] (-3.25,-1.75)--(-3.75,-1.25);
%\draw [-latex] (-3.75,-1.25)--(-4.25,-1.25);
%\draw [-latex] (-4.25,-1.25)--(-4.75,-1.75);
%\draw [-latex] (-4.75,-1.75)--(-4.75,-2.25);
\node at (-4.75,-2.35) {$c_1$};
\node at (-3.35,-1.75) {$c_2$};
\node at (-9.35,-3.7) {$c_1'$};
\node at (0,0) [right] {$x_2$};
\node at (-3.8,-1.95) [below right] {$z$};
\node at (-10,-4) [left] {$x_1$};
\node at (-0.75,-0.3) {$c_2'$};
\node at (0,0) {$\bullet$};
\node at (-4,-2) {$\bullet$};
\node at (-10,-4) {$\bullet$};
\end{tikzpicture}
\caption{Illustration for Lemma~\ref{lem:doubledistance}}\label{fig:foldback}
\end{figure}

Since $\mathrm{dist}(c_1,c_2)=\ell(w_0)$ the affine geometry of the Coxeter complex (c.f. \cite[\S11.5]{AB:08}) implies that if $\gamma_1$ is a minimal length gallery joining $c_1'$ to $c_1$, and $\gamma_2$ is a minimal length gallery joining $c_1$ to $c_2$, and $\gamma_3$ is a minimal length gallery joining $c_2$ to $c_2'$, then the concatenation $\gamma=\gamma_1\cdot\gamma_2\cdot\gamma_3$ is a minimal length gallery joining $c_1'$ to $c_2'$ (see Figure~\ref{fig:foldback}). It follows that all chambers shown in Figure~\ref{fig:foldback} lie in the convex hull of $x_1$ and $x_2$, and hence they lie in a common apartment. It is then clear that $\vect{d}(x_1,x_2)=\vect{d}(x_1,z)+\vect{d}(z,x_2)$. 
%
%
% Since $\delta(d_1,d_2)=m_{\lambda_1^*}$, $\delta(d_2,d_3)=w_0$, and $\delta(d_3,d_4)=m_{\mu}$, the gallery $\gamma=\gamma_1\cdot\gamma_2\cdot\gamma_3$ has type $m_{\lambda_1^*}\cdot w_0\cdot m_{\mu}$. By Lemma~\ref{lem:concatenation} the gallery $\gamma$ is of reduced type, and hence
%$$
%\delta(d_1,d_4)=m_{\lambda_1^*} w_0 m_{\mu}=m_{\lambda_1^*+\lambda_2}.
%$$
%Thus $\vect{d}(x_1,x_2)=\lambda_1^*+\lambda_2=\vect{d}(x_1,z)+\vect{d}(z,x_2)$.
%
%
%Let $\tilde{W}=P\rtimes W_0$ be the \textit{extended affine Weyl group}. Then $\tilde{W}\cong W\rtimes (P/Q)$, where $P/Q$ is a finite group. 
%
%Let $\lambda_j=\vect{d}(z,x_j)$ for $j=1,2$. Choose a minimal length gallery $\gamma_1$ joining $d_1=c(x_1,z)$ to $d_2=c(z,x_1)$, a minimal length gallery $\gamma_2$ joining $d_2=c(z,x_1)$ to $d_3=c(z,x_2)$, and a minimal length gallery $\gamma_3$ joining $d_3=c(z,x_2)$ to $d_4=c(x_2,z)$. Since $\delta(d_1,d_2)=m_{\lambda_1^*}$, $\delta(d_2,d_3)=w_0$, and $\delta(d_3,d_4)=m_{\mu}$, the gallery $\gamma=\gamma_1\cdot\gamma_2\cdot\gamma_3$ has type $m_{\lambda_1^*}\cdot w_0\cdot m_{\mu}$. By Lemma~\ref{lem:concatenation} the gallery $\gamma$ is of reduced type, and hence
%$$
%\delta(d_1,d_4)=m_{\lambda_1^*} w_0 m_{\mu}=m_{\lambda_1^*+\lambda_2}.
%$$
%Thus $\vect{d}(x_1,x_2)=\lambda_1^*+\lambda_2=\vect{d}(x_1,z)+\vect{d}(z,x_2)$.
\end{proof}

For $z\in V_P$, let $\Delta(z)$ denote the set of all chambers $c$ of $\Delta$ with $z\in c$. Note that the maximum numerical distance between chambers $c,c'\in \Delta(z)$ is $\ell(w_0)$ (see Figure~\ref{fig:foldback}). 

\begin{lemma}\label{lem:noise}
Let $x\in V_P$. Suppose that $c_1,c_2\in \Delta(x)$ with $\mathrm{dist}(c_1,c_2)=\ell(w_0)$. Then
$$
|\{c\in \Delta(x)\mid \mathrm{dist}(c_1,c)=\mathrm{dist}(c_2,c)=\ell(w_0)\}|\geq \kappa q^{\ell(w_0)},\quad\text{where $\kappa=(1-q^{-1})^{\ell(w_0)}$}.
$$ 
\end{lemma}

%
%\begin{lemma}\label{lem:noise}
%Let $\Delta$ be an irreducible spherical building with uniform thickness $2\leq q<\infty$. Whenever $c,d\in\Delta$ are opposite chambers we have $|\{e\in \Delta\mid \delta(c,e)=\delta(d,e)=w_0\}|\geq \kappa q^{\ell(w_0)}$ where $\kappa=(1-q^{-1})^{\ell(w_0)}$. 
%\end{lemma}

\begin{proof}
Without loss of generality we may assume that $\tau(x)=0$. Then the fact that $\mathrm{dist}(c_1,c_2)=\ell(w_0)$ implies that $\delta(c_1,c_2)=w_0$ is the unique longest element of $W_0$. In fact $\Delta(x)$ is a spherical building of type $(W_0,S_0)$, and thus the following property holds (see \cite[\S5.3 and \S5.7]{AB:08}): If $c,c'\in\Delta(x)$ with $\mathrm{dist}(c,c')=\ell(w_0)$, then for each $1\leq i\leq n$ there is a unique chamber $p_i(c,c')\in \Delta(x)$ with $c\sim_i p_i(c,c')$ and $\mathrm{dist}(c',p_i(c,c'))=\ell(w_0)-1$. For all other chambers $d\in\Delta(x)$ with $c\sim_i d$ we have $\mathrm{dist}(c',d)=\ell(w_0)$.  

Let $w_0=s_{i_1}\cdots s_{i_N}$ be a reduced expression. Using the above property, there are $q-1$ chambers $d_1\in\Delta(x)$ with $d_1\sim_{i_1}c_1$ and $d_1\neq p_{i_1}(c_1,c_2)$, and all of these chambers satisfy $\mathrm{dist}(c_2,d_1)=\ell(w_0)$. For each of these chambers $d_1$ there are $q-1$ chambers $d_2$ with $d_2\sim_{i_2}d_1$ and $d_2\neq p_{i_2}(d_1,c_2)$, and all of these chambers satisfy $\mathrm{dist}(c_2,d_2)=\ell(w_0)$. Continuing in this way we construct $(q-1)^{\ell(w_0)}$ distinct galleries
$$
c_1\sim_{i_1}d_1\sim_{i_2}d_2\sim_{i_3}\cdots\sim_{i_N}d_N
$$
with $\mathrm{dist}(c_2,d_j)=\ell(w_0)$ for $1\leq j\leq N$. The end chambers $d_N$ of these galleries are all distinct (this follows, for example, from \cite[Proposition~2.1]{Par:06a}), and moreover $\mathrm{dist}(c_1,d_N)=\ell(w_0)$ by construction. Hence the result.
\end{proof}

We now provide analogues of Lemmas~\ref{lem:1} and~\ref{lem:2}. Let $\kappa=(1-q^{-1})^{\ell(w_0)}$, as in Lemma~\ref{lem:noise}.

\begin{lemma}\label{lem:1buildings}
Suppose that $\Delta$ has $(-1)$-type. Let $\E\subseteq V_P$, $k>0$, $\vect{r}\in\mathbb{N}^k$, and $\vect{\la}\in\mathbb{M}^k$. Let $r=\max\{r_i\mid 1\leq i\leq k\}$. If $\E$ contains no balanced $(k,\vect{\la},\vect{r})$-star then for each $v\in S_{\nu-\la}$ with $\nu\geq \la\gg \la_k$ the proportion of the atoms of $\cF_{\nu,\la_1}$ contained in $C(v,\la)$ with the property that they intersect $\E$ in at least $r$ vertices is at most~$1-\kappa$.
\end{lemma}

\begin{proof}
We introduce the following terminology for the proof. A vertex $x\in S_{\nu-\eta}$ (with $0\leq \eta\ll\nu$) is said to have ``type $A$'' if there are distinct chambers $c_0,c_1\ldots,c_{\ell}\in O(x)$, where $\ell=(1-\kappa) q^{\ell(w_0)}$, with the property that $\E\cap C(x,c_j,\eta)$ contains at least $r$ elements for $j=1,\ldots,\ell$, and is said to have ``type $B$'' otherwise. The key observation is that if $x$ has type $A$ then there exist distinct vertices $v_0,v_1^1,\ldots,v_1^r\in S_{\nu}$ with $\vect{d}(v_0,v_1^j)=2\eta$ for $j=1,\ldots,r$. To see this, choose $v_0\in \E\cap C(x,c_{0},\eta)$. By Lemma~\ref{lem:noise} there is at least one index $1\leq j\leq \ell$ such that $\delta(c_0,c_j)=w_0$. Then by Lemma~\ref{lem:doubledistance}, using the fact that $\Delta$ has $(-1)$-type, we have $\vect{d}(v_0,z)=\vect{d}(v_0,x)+\vect{d}(x,z)=\eta^*+\eta=2\eta$ for all $z\in C(x,c_j,\eta)$, and so any choice of $v_1^1,\ldots,v_1^r\in \E\cap C(x,c_j,\eta)$ will work.

%
%We make the following observation. Let $0<s' < s<n$. If $x\in S_{n-s}$ then 
%$$
%C(x,s)=\bigsqcup_{y\in C(x,1)}C(y,s-1),
%$$
%and for each $y\in C(x,1)$ the set $C(y,s-1)$ is a union of $q^{s-s'-1}$ atoms of $\cF_{n,s'}$. In particular, each set $C(y,s-1)$, $y\in C(x,1)$, contains the same number of $\cF_{n,s'}$ atoms, and so if $x\in S_{n-s}$ has type $B$ then the proportion of the $\cF_{n,s'}$ atoms of $C(x,s)$ containing at least $r$ elements of $\E$ is at most~$\kappa$.

Let $v\in S_{\nu-\la}$. Let $\mathcal{X}$ be the set of all sequences $(x_k,x_{k-1},\ldots,x_1)$ such that $x_k\in S_{\nu-\la_k}\cap C(v)$ and $x_{k-j}\in C(x_{k-j+1})\cap S_{\nu-\la_{k-j}}$ for $j=1,2,\ldots,k-1$. Suppose that there exists a sequence $(x_k,x_{k-1},\ldots,x_1)\in\mathcal{X}$ such that each $x_j$ has type~$A$. Thus, as noted above, there are chambers $c_1,c_1'\in O(x_1)$, and vertices $v_0\in \E\cap C(x_1,c_1,\la_1)$ and $v_1^1,\ldots,v_1^r\in \E\cap C(x_1,c_1',\la_1)$ with $\vect{d}(v_0,v_1^i)=2\la_1$ for all $1\leq i\leq r$. 

For each $j=2,\ldots,k$ let $c_j\in O(x_j)$ be the chamber $c_j=c(x_1,x_j)$. Since $x_j$ has type~$A$, the argument above shows that there is a chamber $c_j'\neq c_j$ in $O(x_j)$ and distinct vertices $v_j^1,\ldots,v_j^r\in \E\cap C(x_j,c_j',\la_j)$ with $d(v_0,v_j^i)=2\la_j$ for each $1\leq i\leq r$ and $1\leq j\leq k$, and so $\{v_0\}\cup\{v_j^i\mid 1\leq i\leq r,\,1\leq j\leq k\}$ forms a balanced $(k,\vect{\la},r\vect{1})$-star, where $\vect{1}\in\mathbb{N}^k$ is the vector with every entry equal to~$1$. In particular $\E$ contains a balanced $(k,\vect{\la},\vect{r})$-star, a contradiction.

Thus every $(x_k,x_{k-1},\ldots,x_1)\in\mathcal{X}$ contains at least one vertex of type~$B$. Note that if $x_j$ has type~$B$ then the proportion of $\cF_{\nu,\la_1}$ atoms of $C(x_j,\la_j)$ intersecting $\E$ in at least $r$ vertices is at most~$1-\kappa$. Thus we can partition the set of $\cF_{\nu,\la_1}$ atoms in $C(v,\la)$ in such a way that in each part of the partition the proportion of atoms with the property that they intersect $\E$ in at least $r$ vertices is at most~$1-\kappa$. Thus the proportion of all $\cF_{\nu,\la_1}$ atoms of $C(v,\la)$ with this property is at most~$1-\kappa$, and hence the result.
\end{proof}

\begin{lemma}\label{lem:2buildings}
Suppose that $\Delta$ has $(-1)$-type. Let $\E\subseteq V_P$, $k>0$, and $\vect{r}\in\mathbb{N}^k$. For $1\leq j\leq \ell$ let $\vect{\la}_j=(\la_{i,j})_{i=1}^k\in \mathbb{M}^k$, and suppose that $\la_{k,j}\ll\la_{1,j+1}$ for each $j=1,\ldots,\ell-1$. If $\E$ contains no balanced $(k,\vect{\la}_j,\vect{r})$-stars for each  $1\leq j\leq \ell$ then for all $\nu\gg\la_{k,\ell}$ we have
$$
\frac{|\E\cap S_{\nu}|}{|S_{\nu}|}<(1-\kappa)^{\ell}+rq^{\ell(w_0)-\langle\la_{1,1},2\rho\rangle},
$$
where $r=\max\{r_i\mid 1\leq i\leq k\}$. 
\end{lemma}

\begin{proof}
Lemma~\ref{lem:1buildings} (applied to the case $v=o$) implies that the proportion of $\cF_{\nu,\la_{1,\ell}}$ atoms intersecting $\E$ in at least $r$ vertices is at most $1-\kappa$. Let $X_{\ell}$ be such an atom, and let $x_{\ell}$ be the projection of this atom onto $S_{\nu-\la_{1,\ell}}$ (that is, $X_{\ell}=S_{\nu}\cap C(x_{\ell})$). Lemma~\ref{lem:1buildings}, this time applied to the case $v=x_{\ell}$, implies that the proportion of the $\cF_{\nu,\la_{1,\ell-1}}$ atoms contained in $X_{\ell}$ with the property that they intersect $\E$ in at least $r$ vertices is at most~$1-\kappa$. Hence the proportion of all $\cF_{\nu,\la_{1,\ell-1}}$ atoms with the property that they intersect $\E$ in at least $r$ vertices is at most~$(1-\kappa)^2$. Iterating this process shows that the proportion of all $\cF_{\nu,\la_{1,1}}$ atoms with the property that they intersect $\E$ in at least $r$ vertices is at most~$(1-\kappa)^{\ell}$. Each atom in the remaining $1-(1-\kappa)^{\ell}$ proportion of atoms contains at most $r-1$ elements of $\E$. Since the total number of $\cF_{\nu,\la_{1,1}}$ atoms is $|S_{\nu-\la_{1,1}}|q^{\ell(w_0)}=q^{\ell(w_0)-\langle\la_{1,1},2\rho\rangle}|S_{\nu}|$ (see Corollary~\ref{cor:sphere} and Lemma~\ref{lem:atoms}) we have
\begin{align*}
\frac{|\E\cap S_{\nu}|}{|S_{\nu}|}&\leq(1-\kappa)^{\ell}+(1-(1-\kappa)^{\ell})q^{\ell(w_0)-\langle\la_{1,1},2\rho\rangle}(r-1)<(1-\kappa)^{\ell}+rq^{\ell(w_0)-\langle\la_{1,1},2\rho\rangle},
\end{align*}
hence the result.
\end{proof}

%
%For $\E\subseteq V_P$ we define the density by
%\begin{align*}
%d^*(\E)=\limsup_{n\to\infty}\frac{1}{n}\sum_{k=0}^n\frac{|S_n\cap \E|}{|S_n|}.
%\end{align*}
%(we might use $\liminf$ instead I guess).
%

\begin{proof}[Proof of Theorem~\ref{thm:main1buildings}] The proof follows from Lemmas~\ref{lem:1buildings} and~\ref{lem:2buildings} in exactly the same fashion as the proof of Theorem~\ref{thm:main1}.
\end{proof}

\begin{remark} We note the following easy extensions of Theorem~\ref{thm:main1buildings}.
\begin{compactenum}[$(1)$]
\item We have assumed that our buildings (and trees) have uniform thickness parameter~$q$. Our techniques apply more generally to the case of locally finite thick ``regular'' affine buildings. These buildings have the property that for each $0\leq i\leq n$ the cardinality $q_i=|\{d\in\Delta\mid d\sim_i c\}|$ is independent of $c\in\Delta$, and $2\leq q_i<\infty$ for all $0\leq i\leq n$. In the $\tilde{A}_1$ case these buildings are ``bi-regular'' trees, where the valencies alternate according to the bipartite structure of the tree.
\item In the definition of $\mathbb{M}^k$ in Section~\ref{subsec:5a} we used $\ll$. Less restrictively one could instead use $<$ in this definition. Theorem~\ref{thm:main1buildings} still holds using this less restrictive definition, with a very similar proof. However some technical changes are required in Lemmas~\ref{lem:1buildings} and~\ref{lem:2buildings}, for if $z,z'\in V_P$ with $\vect{d}(z,z')\in P^+\backslash P^{++}$ then the chamber $c(z,z')$ from Figure~\ref{fig:start} is no longer unique. Instead one must argue using the (unique) ``projection'' of $z'$ onto the ``residue'' of~$z$, which in general is a lower dimensional simplex. While this makes the arguments more technical, the essential details are the same.
\end{compactenum}
\end{remark}

\begin{remark}\label{rem:conj}
We have not been able to push our techniques through to the case of non-$(-1)$-type affine buildings (note that $\la^*=\la$ is used in an essential way in Lemma~\ref{lem:1buildings}). Indeed Theorem~\ref{thm:main1buildings} does not hold in its current form for non-$(-1)$-type buildings. For example, in a thick $\tilde{A}_n$ building the set $\E=V_Q$ has $d^*(\E)=1>0$, and we have $\vect{d}(x,y)\in Q\cap P^+$ for all $x,y\in \E=V_Q$. Thus if $\la=\omega_1+N\rho$ then $2\la\notin Q$ (assuming that $n\geq 2$). Thus there are arbitrarily large $\la\in P^+$ such that $2\la\notin\{\vect{d}(x,y)\mid x,y\in\E\}$. 
\end{remark}

In light of Remark~\ref{rem:conj}, we make the following conjecture.

\begin{conj}\label{conj:main}
Let $\Delta$ be an irreducible affine building with uniform thickness $2\leq q<\infty$. Let $\E\subseteq V_P$ with $d^*(\E)>0$. For each $k>0$ there exists $K=K(\E,k)>0$ such that whenever $\la_1,\ldots,\la_k\in Q\cap P^+$ with $\la_k\geq \cdots\geq \la_2\geq \la_1\geq K\rho$ there exists a subset $\{v_0,v_1,\ldots,v_k\}\subseteq \E$ with $\vect{d}(v_0,v_j)=\la_j$ for all $1\leq j\leq k$ and $\vect{d}(o,v_0)=\vect{d}(o,v_1)=\cdots=\vect{d}(o,v_k)$. 
\end{conj}

\subsection{Application to $p$-adic Lie groups}\label{subsec:6}

We conclude with an application to sets of positive density in $p$-adic Lie groups. Let $\FF$ be a local field with valuation ring $\fo$ and residue field $\FF_q$, and let $G=G(\FF)$ be a Chevalley group with root system $\Phi$. Let $K=G(\fo)$. There is an affine building $\Delta(G,K)$ associated to $(G,K)$ whose set of type~$0$ vertices is $G/K=V_Q\subseteq V_P$ (see \cite{BT:72}). This building has uniform thickness parameter~$q$. 

There are elements $\varpi_{\la}\in G$ such that 
$$
G=\bigsqcup_{\la\in Q\cap P^+}K\varpi_{\la}K
$$  
(roughly speaking, $\varpi_{\la}$ is a diagonal matrix whose entries are powers of the uniformiser~$\pi$). Moreover, the vector distance between vertices $gK$ and $hK$ is $\vect{d}(gK,hK)=\lambda$ if and only if $g^{-1}h\in K\varpi_{\la}K$. For $g\in G$ let $\la(g)=\vect{d}(K,gK)$.

We define upper density $d^*(\E)$ of a subset $\E\subseteq G/K$ as in~(\ref{eq:defdensity}).

\begin{cor}\label{cor:padic}
Let $(G,K)$ be as above, and suppose that $\Delta(G,K)$ has $(-1)$-type. Let $r\in\NN$. Let $\E\subseteq G/K$ with $d^*(\E)>0$. There exists a constant $N>0$ such that for all $g_1,\ldots,g_r$ with $\la(g_j)\in 2P$ and $N\rho\ll\la(g_1)\ll\cdots\ll\la(g_r)$ we have 
$$
\E\cap \E g_1K\cap\cdots\cap \E g_rK\neq\emptyset.
$$
\end{cor}

\begin{proof}
Note that $\E\cap \E g_1K\cap\cdots\cap \E g_rK\neq\emptyset$ if and only if there exit $h_0K,\ldots,h_rK\in \E$ with $h_0K\subseteq h_jKg_jK$ for all $1\leq j\leq r$. Thus $h_j^{-1}h_0\in Kg_jK=K\varpi_{\la(g_j)}K$, and so $\vect{d}(h_jK,h_0K)=\la(g_j)$ for all $1\leq j\leq r$. Since $\Delta(G,K)$ is of $(-1)$-type we have $\vect{d}(h_0K,h_jK)=\la(g_j)$, and so $\E\cap \E g_1K\cap\cdots\cap \E g_rK\neq\emptyset$ if and only if there exist vertices $v_j=h_jK\in \E$ with $\vect{d}(v_0,v_j)=\la(g_j)$ for $1\leq j\leq r$. The result follows from Theorem~\ref{thm:main1buildings}.
\end{proof}

\end{document}